\documentclass[pdflatex,sn-mathphys,authoryear]{sn-jnl}

\usepackage{natbib}
\usepackage{graphicx}%
\usepackage{multirow}%
\usepackage{amsmath,amssymb,amsfonts}%
\usepackage{amsthm}%
\usepackage{mathrsfs}%
\usepackage[title]{appendix}%
\usepackage{xcolor}%
\usepackage{textcomp}%
\usepackage{manyfoot}%
\usepackage{booktabs}%
\usepackage{algorithm}%
\usepackage{algorithmicx}%
\usepackage{algpseudocode}%
\usepackage{listings}%
\usepackage{stmaryrd}
\usepackage{enumerate}
\usepackage{tikz}
\usepackage{pdflscape}
\usepackage{geometry}
\usepackage{caption}
\usepackage{tabularx}

\newcommand{\rev}[1]{#1}
\newtheorem{theorem}{Theorem}[section]

\newtheorem{proposition}[theorem]{Proposition}
\newtheorem{example}[theorem]{Example}

\theoremstyle{definition}
\newtheorem{definition}[theorem]{Definition}
\newtheorem{remark}[theorem]{Remark}

\newcommand{\N}{\mathbb{N}}

\newcommand{\R}{\mathbb{R}}

\newcommand{\FF}{\mathcal{F}}

\newcommand{\TT}{\mathcal{T}}

\newcommand{\dint}{{\rm d}}

\newcommand{\HH}{\mathcal{H}}

\newcommand{\I}{\mathbb{I}}

\newcommand{\RV}{\operatorname{RV}}

\begin{document}

\title{Stein's method for Fr\'echet approximation: a regularly varying
  functions approach}

\author*[1,2]{\fnm{Paul}\sur{Mansanarez} }\email{paul.mansanarez@ulb.be}
\author[1]{\fnm{Guillaume}\sur{Poly} }\email{guillaume.poly@univ-nantes.fr}
\author[2, 3]{\fnm{Yvik}\sur{Swan}}\email{yvik.swan@ulb.be}

\affil[1]{\orgname{Nantes Universit\'e}, \country{France}}
\affil[2]{\orgname{Universit\'e libre de Bruxelles}, \country{Belgium}}
\affil[3]{\orgname{Vrije Universiteit  Brussel},   \country{Belgium}}

\abstract{ We develop a variant of Stein's method of comparison of generators to bound the Kolmogorov, total variation, and Wasserstein-1 
  distances between distributions on the real line. Our discrepancy is
  expressed in terms of the ratio of reverse hazard rates; it
  therefore remains tractable even when density derivatives are
  intractable. Our main application concerns the approximation of
  normalized extremes by Fr\'echet laws. In this setting, the new
  discrepancy provides a quantitative measure of distributional
  proximity in terms of the average regular variation at infinity of
  the underlying cumulative distribution function. We illustrate the
  approach through explicit computations for maxima of Pareto, Cauchy,
  and Burr~XII distributions. Our new discrepancy also opens the way to statistical applications which we outline.}


\maketitle

\section{Introduction}
The \emph{Fr\'echet law} is a continuous probability measure on the
real line with cumulative distribution function (cdf) and density
(pdf) given by
\[
  \Phi_\alpha(x) = e^{-x^{-\alpha}} \, \mathbb{I}_{\rev{\R_{>0}}}(x),
  \quad \phi_\alpha(x) = \alpha x^{-(1+\alpha)} e^{-x^{-\alpha}} \,
\mathbb{I}_{\rev{\R_{>0}}}(x),
\]
where $\alpha > 0$ ($\mathbb{I}_A$ is the indicator function of the
set $A$ and   \rev{  $\R_{> 0}$ (resp. $\R_{\ge0}$ ) denotes the set of positive (resp. nonnegative) real numbers).} It is one of the three \emph{generalized extreme value
  (GEV) laws}, which are the only possible limits for suitably
normalized maxima of independent and identically distributed real
random variables. The two others are the \emph{Gumbel} and
\emph{Weibull} laws.

Convergence of extremes towards a Fr\'echet \rev{distribution} (or, more generally,
towards \rev{a distribution in the GEV class}) is a well-studied problem. We say that a cdf $F$
belongs to the \emph{domain of attraction} of $\Phi_{\alpha}$, and
write $F \in \mathrm{DA}(\alpha)$, when there exist real sequences
$(a_n)$ and $(b_n)$, with $a_n>0$, such that
\begin{equation}
  \label{eq:2}
  \rev{F_n(x) := \Bigl(F(a_n x+b_n)\Bigr)^n  \longrightarrow \Phi_{\alpha}(x)
  \mbox{ as } n \to \infty.}
\end{equation}
The Fr\'echet domain of attraction is completely determined by the
variation of $F$ at infinity: $F \in \mathrm{DA}(\alpha)$ if and only
if $F(x)<1$ for all $x \in \mathbb{R}$ and
 \begin{equation}
   \label{eq:3}
   \lim_{t \to \infty} \frac{1 - F(tx)}{1-F(t)} = x^{-\alpha} \mbox{
     for all } x >0;
 \end{equation}
 in this case we say $1-F$ is regularly varying at infinity with index
 $-\alpha$, and write $1-F \in \mathrm{RV}_{-\alpha}$. Furthermore one
 can take \rev{$b_n=0$ and $a_n \sim F^{\leftarrow}(1-1/n)$} (here
 $F^{\leftarrow}$ is the generalized inverse, a.k.a.\ quantile function
 of $F$). See Appendix~\ref{sec:regul-vary-funct} for definitions and
 references as well as an overview of the Karamata theory which
 underpins these developments.

The dependence on the sample size in the approximation in
\eqref{eq:2} has  been studied in the literature.  Rates of
convergence in the Kolmogorov distance -- denoted
$\mathrm{Kol}(F_n, \Phi_{\alpha})$  --
and in the total variation distance -- denoted
$\mathrm{TV}(F_n, \Phi_{\alpha})$ -- can be obtained,
e.g.\ from \cite{smith1982uniform} or \cite[Section
2.4]{resnick2008extreme}. In addition, more refined results involving so-called 
second-order variation conditions (see Remark \ref{rem:second-order} for a definition)  are available in \cite{Resnick1996,
resnick2008extreme, Cheng2001}, which not only provide convergence
rates but also identify the limiting constants via certain
non-trivial functionals depending on the underlying distribution
$F \in \mathrm{DA}(\alpha)$.  \rev{Second-order variation also appears
in recent Wasserstein and coupling approaches to extreme-value
approximation; see, for instance, \cite{BobbiaDombryVarron2021}.}  However, these results can be technically demanding to apply in explicit
examples. Table~\ref{tab:rest} gives a few illustrative benchmark cases;
a more systematic treatment will appear in \cite{benmou25}.

\begin{table}
\centering\tiny
\renewcommand{\arraystretch}{1} \setlength{\tabcolsep}{6pt}
\begin{tabular}{|l|c|c|c|c|c|c|}
  \hline
  \textbf{Distr.} & \( F(x) \) & Supp. & DA & \( a_n\) & \( c_n \)
  & $\ell$\\
  \hline
  Pareto & \( 1 - x^{-\alpha} \) & \( [1, \infty) \)  & \( \alpha \) & \( n^{1/\alpha} \) & \(  n \) & \( 2e^{-2} \) \\
  Cauchy & \( \dfrac{1}{\pi} \arctan(x) + \dfrac{1}{2} \) & \(\mathbb{R}\)  & \(1\)& \( \dfrac{n}{\pi} \) & \( n \)  & \( 2e^{-2} \)  \\
  Log-logistic & \( \dfrac{x^\alpha}{1 + x^\alpha} \) & \( (0, \infty)\) & \( \alpha\)  & \(n^{1/\alpha}\) & \(n\) & \( 2e^{-2} \)  \\
  GPD & \( 1 - \left(1 + \xi \dfrac{x}{\sigma} \right)^{-1/\xi} \)  (\( 0< \xi < 1 \))  & \([0,\infty)\) & \( 1/\xi\)   & \( n^\xi \) & \(  n \) & \(2e^{-2}\)\\
  / & \( 1 - x^{-\alpha} - x^{-\alpha - \beta} \)  (\(\beta >\alpha\)) & \(  [z_{\alpha, \beta}, \infty ) \) & \( \alpha\) & \rev{\(n^{1/\alpha}\)} & \( n \) & \(2e^{-2}\)  \\
  / & \( 1 - \dfrac{x^{-\alpha}}{\log(\log x)} \) & \((e^e,\infty)\) & \(\alpha\) & \( \left(\dfrac{n}{\log(\log n)}\right)^{1/\alpha} \) & \(\log(\log n)\) & /   \\
  / & \( 1 - \dfrac{x^{-\alpha}}{1+\log x} \) & \( [1,\infty)\)& \(\alpha\) & \( \left( \dfrac{n}{\log n} \right)^{1/\alpha} \) & \(\dfrac{\log n}{\log(\log n)} \) & /   \\
  \hline
\end{tabular}

\vspace{1em}

\caption{\emph{Normalizing sequences $(a_n)$ (we take $b_n=0$), rates
    of convergence $c_n$ and limits
    $\ell=\lim_{n\to\infty}c_n\mathrm{Kol}(F_n,\Phi_\alpha)$. The limits in
    the last two lines are known but inelegant, hence not reported. \rev{The number $z_{\alpha,\beta}$ is the unique solution of the equation $1+x^{\beta}=x^{\alpha+\beta}$ on $\R_{\geq0}$.} }
    \label{tab:rest}}
 \end{table}

Among the many available alternative methods for obtaining quantitative fixed-$n$
approximations, Stein's method has emerged as a flexible and powerful tool. While its applicability and effectiveness
have been demonstrated across a wide range of approximation problems
(see the websites in
footnote\footnote{\url{https://sites.google.com/site/steinsmethod} and
  \url{https://sites.google.com/site/malliavinstein/home}} for an
overview), its use in the context of GEV approximation remains rather
limited, with only a few publications addressing this problem
directly.  Notable contributions include \cite{feidt2013stein}, which
applies a version of the Stein--Chen Poisson approximation,
\cite{costaceque2024stein, CostacequeDecreusefond2025}, which develops a Markov semigroup--based
framework, and \cite{Kusumoto2020}, which employs the generator
comparison strategy first introduced in \cite{bartholme2013rates}. The
generator comparison strategy \cite{bartholme2013rates,Kusumoto2020}
is effective but only in relatively simple settings, as it relies on
comparisons of \emph{score functions}, i.e.\ derivatives of the
log-likelihood ratio, which are not tractable in the GEV setting. This
limitation is precisely where the present work intervenes. 

\rev{In this paper, we build on the Stein-operator framework of
\cite{ley2017stein} to develop a  version of the generator comparison method which avoids differentiating
densities. The resulting approach connects the distributional proximity
in~\eqref{eq:2} to a quantitative form of regular variation at infinity for the
cdf \(F\), or equivalently to a quantitative counterpart of~\eqref{eq:3}; see
Section~\ref{sec:mainresn} for the precise statements.

Our contribution has three complementary aspects. First, we identify a new Stein
discrepancy which provides a unified mechanism for deriving Kolmogorov, total
variation, and Wasserstein-type bounds. As illustrated in the examples in Section \ref{sec:mainresn},
this discrepancy is explicit and tractable in many classical situations, leading
to precise rates of convergence. Second, from the point of view of Stein's
method itself, our construction yields Stein operators which avoid
derivatives of densities; this reverse-hazard formulation appears to be new
and may be of interest beyond extreme-value approximation.
Finally, as is often the case with Stein discrepancies, the resulting quantity
can also be repurposed beyond convergence bounds. Inspired by
\cite{eetal2025}, we indicate in Examples~\ref{ex:bootstrap-diagnostic}
and~\ref{ex:estimation} how it may be used as a diagnostic for tail regularity
and for estimation procedures. 
}
  
\rev{The rest of the paper is organized as follows. Section~\ref{sec:mainresn}
contains the main abstract results and several illustrative examples; the
computations for these examples are mostly routine and are deferred to
Appendix~\ref{sn:appendixfurther}. Section~\ref{sec:steins-meth-comp}
presents the Stein-method argument leading to Theorems~\ref{mainresult}
and~\ref{mainresultwass}, where the new Stein discrepancy is introduced.
This discrepancy bounds distributional distances through an averaged comparison
of reverse hazard rates. Section~\ref{sec:further-proofs} contains the proof of
Theorem~\ref{borneXalpha}, showing that the discrepancy indeed provides a
quantitative form of (strong) Fr\'echet-domain regularity. Finally,
Appendix~\ref{sec:regul-vary-funct} collects the background material on extreme
value theory and regular variation needed in the proofs.}

\section{Main results and examples}
\label{sec:mainresn}

To state our results we start with some notations and definitions.
First of all, if $F$ is a cdf on the real line, we write $\overline F(x) = 1-F(x)$
  for the survival function of $F$. If $h$ is a test
function, we denote
\begin{equation*}
  F(h) = \int_{-\infty}^{+\infty} h(x)\, d F(x)
\end{equation*}
the expectation of $h$ under $F$ (provided the integral exists and is
finite). Next, given two cdfs $P, Q$, the \emph{integral probability
  metrics} from $P$ to $Q$ are the functionals
\begin{align*}
  \mathcal{D}(P, Q \mid \mathcal{H})   &  =  \sup_{h \in \mathcal{H}} | (P-Q)(h)|,
\end{align*}
where $\mathcal{H}$ is a well-chosen class of test functions. The
Kolmogorov distance corresponds to $\mathcal{H} = \mathcal{H}_K$ the
collection of indicators $\mathbb{I}_{(-\infty, z]}$ of half lines on
$\mathbb{R}$, the total variation distance to
$\mathcal{H} = \mathcal{H}_T$ the collection of indicators
$\mathbb{I}_B$ of Borel sets on $\mathbb{R}$. We will also consider
the Wasserstein-1 (a.k.a.\ Kantorovitch-Rubinstein) distance
$\mathrm{Wass}(P, Q)$, which is obtained with
$\mathcal{H} = \mathrm{Lip}(1)$ the set of Lipschitz functions on
$\mathbb{R}$ with Lipschitz constant 1.

Throughout the paper we will restrict our attention to cdfs
satisfying the following assumption.

\medskip
\noindent \textbf{Assumption 0:} \emph{ There exist a starting mass
  $f_{0} \geq 0$, a left endpoint $c_{F} \in [-\infty,+\infty)$, and a function $f : \mathbb{R} \to \R_{\geq 0}$,
  integrable, such that the cumulative distribution function $F$ is
  given by
\begin{equation}
  \label{eq:4}
F(x) =
\begin{cases}
 0, & x < c_{F}, \\[6pt]
f_{0} + \displaystyle\int_{c_{F}}^{x} f(u)\,\dint u, & x \geq c_{F}.
\end{cases}
\end{equation}
Moreover, $f$ is continuous and strictly positive on
$(c_{F},+\infty)$. \rev{This density-like function $f$ will be referred in the sequel as the pdf of the distribution with cdf $F$}.}  \medskip

 Note how Assumption 0 guarantees that the distribution has
no mass below $c_F$, may have a point mass of size $f_0$ at $x = c_F$,
and has continuous and strictly positive density $f$ on
$(c_F, + \infty)$. In particular, this ensures that $F$ is strictly
increasing in its right tail, has no mass at infinity (i.e.,
$\lim_{x \to + \infty} F(x) = 1$) and also $F(x) < 1$ for all
$x < +\infty$.

\begin{remark}
\rev{   Assumption~0 is mainly a technical regularity condition ensuring that the
Stein operators used below are well defined and that the integration by
parts argument has no additional boundary terms, except for the possible
atom at the left endpoint. It covers the standard absolutely continuous
heavy-tailed models considered in this paper, as well as distributions
with a lower endpoint atom and a positive continuous density thereafter.
It is not intended as a modelling assumption for raw data.  The
condition could  be relaxed in several directions. For instance,
one may allow densities which vanish at isolated points, several atoms,
or piecewise smooth densities, at the price of adding the corresponding
boundary and jump terms in the Stein identity. }
\end{remark}

\rev{Finally,  given $F$ a cdf satisfying Assumption 0, we define
  \begin{equation*}
    r_F(x)  = \frac{f(x)}{F(x)} \mathbb{I}_{(c_F, +\infty)}(x)   =
    {(\log F)'(x)}\mathbb{I}_{(c_F, +\infty)}(x);
  \end{equation*}
  this function is sometimes referred to as the \emph{reverse hazard
    rate function} of $F$. }

\rev{ Our main abstract results, obtained through a version of Stein's
  method which we shall detail in Section~\ref{sec:steins-meth-comp},
  read as follows.}

  \begin{theorem} \label{mainresult} Let $P, Q$ be two cdfs satisfying
    Assumption 0, with pdfs $p, q$ and supports $[c_P, +\infty)$ and
    $[c_Q, +\infty)$, respectively. Suppose that $P$ has no atom
    (i.e.\ $p_0 = 0$) and that $c_Q\ge c_P\ge -\infty$.  Introduce the
    \emph{Stein discrepancy} \begin{equation}
  \label{eq:7}
  \Delta(Q\mid P) := \int_{c_Q}^{+\infty} \left| 1 -
    \frac{r_P(x)}{r_Q(x)}  \right|  q(x) \, \dint x.
        \end{equation}
Then
        \begin{equation}
    \label{eq:1}
    \mathrm{Kol}(P, Q) \le   \Delta(Q\mid P)  + q_0 \mbox{ and } \mathrm{TV}(P,
    Q) \le 2   \Delta(Q\mid P) + q_0
  \end{equation}
  with $q_0 = Q(\left\{ c_Q \right\})$ the starting mass of $Q$.
\end{theorem}

\begin{theorem} \label{mainresultwass} In addition to the assumptions
  from Theorem \ref{mainresult}, suppose also that $P, Q$ have finite
  mean, that $c_P \ge 0$ and that $q_0=0$. Set
  \begin{equation}
  \label{eq:9}
  \Delta_w(Q\mid P) := \int_{c_Q}^{+\infty} x \left| 1 -
    \frac{r_P(x)}{r_Q(x)}  \right|  q(x)\, \dint x.
        \end{equation}
Then
        \begin{equation}
\label{eq:13}
\mathrm{Wass}(P, Q) \le   2 \mu
\Delta(Q \mid P)  + 3 \Delta_w(Q\mid P)
\end{equation}
($\mu$ is the mean of $P$).
\end{theorem}

The bounds in \eqref{eq:1} and \eqref{eq:13} can be applied for the
comparison of any two distributions satisfying Assumption~0. In particular  they can be used to compare Fr\'echet laws in the Kolmogorov, total variation and Wasserstein-1 distances, as we now illustrate. 

\begin{example}[Comparison of Fr\'echet laws]
\label{sec:compafre}
When $P,Q$ are both Fr\'echet distributions, the discrepancies can be
computed explicitly.  If $\beta>\alpha$, then
\[
  \Delta(\Phi_\alpha\mid\Phi_\beta)
  =
  1
  -2e^{-\left(\frac{\alpha}{\beta}\right)^{\frac{\alpha}{-\alpha+\beta}}}
  -\Gamma\left(\frac{\alpha+\beta}{\alpha}\right)
  +\frac{2\alpha}{\beta}
  \Gamma\left(
    \frac{\alpha}{\beta},
    \left(\frac{\alpha}{\beta}\right)^{-\frac{\beta}{\alpha-\beta}}
  \right).
\]
If $\alpha>1$, then
\[
\begin{aligned}
  \Delta_w(\Phi_\alpha\mid\Phi_\beta)
  =
  -\frac1\alpha\Bigg\{
  &\Gamma\left(-\frac1\alpha\right)
  +\beta\,\Gamma\left(\frac{-1+\beta}{\alpha}\right) \\
  &+2\alpha\,
  \Gamma\left(
    \frac{-1+\alpha}{\alpha},
    \left(\frac{\alpha}{\beta}\right)^{\frac{\alpha}{-\alpha+\beta}}
  \right)  \\
  &-2\beta\,
  \Gamma\left(
    \frac{-1+\beta}{\alpha},
    \left(\frac{\alpha}{\beta}\right)^{\frac{\alpha}{-\alpha+\beta}}
  \right)
  \Bigg\}.
\end{aligned}
\]
Here $\Gamma(x)$ is the gamma function and $\Gamma(x,y)$ the incomplete
gamma function.  A sharper and much simpler upper bound for the
Kolmogorov distance between two Fr\'echet distributions is given in
\cite[Lemma~2.11]{resnick2008extreme}.  That bound, however, is
specific to the Kolmogorov distance, whereas the present estimates also
apply to total variation and Wasserstein-1 distances. Numerical evaluations and further computations can be found in the  supplementary material \citep{mansanonline}. 
\end{example}

Our purpose is to use the bounds  towards the quantitative approximation of the distribution of a maximum by a Fr\'echet distribution. The Fr\'echet
reverse hazard rate is
\begin{equation*}
  r_{\Phi_{\alpha}}(x) = \frac{\alpha }{x^{\alpha+1}} \mathbb{I}_{(0,
    +\infty)}(x).
\end{equation*}
Now let $F$ be a cdf satisfying Assumption 0 and consider $F_n$ as
defined in equation \eqref{eq:2}, with $b_n=0$. Clearly $F_n$ also
satisfies Assumption 0, with support lower bounded by
$c_{F_n} = a_n^{-1}c_F$ and reverse hazard rate
\begin{equation*}
  r_{F_n}(x) = n a_n r_F(a_n x).
\end{equation*}
If $c_F \le 0$, the Fr\'echet distribution has the smallest support so
we take $Q = \Phi_{\alpha}$ and $P = F_n$ in \eqref{eq:7} so that
\begin{equation} \label{eq:1alphansupp2small}
  \Delta(\Phi_{\alpha}\mid F_n) = \int_0^{+\infty} \left| 1-
    \frac{ x^{\alpha+1}}{\alpha}n a_n {r_F(a_nx) }
    \right|\phi_{\alpha}(x)\, \dint x.
  \end{equation}
  If $c_F \ge 0$ then the situation is reversed: we take $Q = F_n$ and
  $P = \Phi_{\alpha}$ in \eqref{eq:7} and the Stein discrepancy
  becomes
  \begin{equation} \label{eq:1alphansupp1small} \Delta(F_n\mid
    \Phi_{\alpha}) = \int_{c_F/a_n}^{+\infty}\left|1-
      \frac{\alpha}{x^{\alpha+1}} \frac{1}{na_n r_F(a_n x)} \right|
    f_{n} (x) \, \dint x.
  \end{equation}
  The corresponding expressions for \eqref{eq:9} are easy to deduce
  from here.

  The Stein discrepancies \eqref{eq:1alphansupp2small} and
  \eqref{eq:1alphansupp1small} have a natural interpretation in the
  extreme value context. \rev{Formally, if the reverse hazard rate satisfies
  $r_F\in\mathrm{RV}_{-\alpha-1}$ and if $n\overline F(a_n)\to1$, then
  Karamata's theorem (see Appendix~\ref{sec:regul-vary-funct}) gives $n a_n r_F(a_n)\to\alpha$, and hence, for
  each fixed $x>0$,}
  \begin{equation}\label{eq:11}
    \rev{n a_n r_F(a_{n} x) \sim n a_n r_F(a_n)x^{-\alpha-1}
    \sim \alpha x^{-\alpha-1},}
  \end{equation}
\rev{where $\sim$ indicates asymptotic equivalence.  The discrepancies measure an averaged version of the error in this
  asymptotic relation. They can thus be viewed as quantitative proxies
  of regular variation at the level of the reverse hazard rate.  Under
  the regularity assumptions stated below, convergence of our Stein
  discrepancy to $0$ provides a quantitative sufficient condition for
  the Fr\'echet approximation.}

\begin{theorem}\label{borneXalpha}
  Let $F$ satisfy Assumption~0 and let $F_n$ be defined through
  \eqref{eq:2} with $b_n=0$ and $(a_n)_n$ a sequence of positive real
  numbers.
  \begin{itemize}
  \item \rev{Assume that $c_F\le 0$, that $F\in\mathrm{DA}(\alpha)$,
    that $r_F\in\mathrm{RV}_{-\alpha-1}$, and that $r_F$ is bounded on
    finite intervals of $\R_{\geq 0}$. Then
    $\Delta(\Phi_\alpha\mid F_n)\to0$ as soon as
    $n\overline F(a_n)\to1$. If $\alpha>1$, then also
    $\Delta_w(\Phi_\alpha\mid F_n)\to0$.}

  \item \rev{Assume that $c_F\ge0$, that $F\in\mathrm{DA}(\alpha)$,
    that $f\in\mathrm{RV}_{-\alpha-1}$, and that $1/r_F$ is bounded on
    finite intervals of $\R_{\geq 0}$. Then, for
    $a_n=F^{\leftarrow}(1-1/n)$, one has
    $\Delta(F_n\mid\Phi_\alpha)\to0$. If $\alpha>1$, then also
    $\Delta_w(F_n\mid\Phi_\alpha)\to0$.}
  \end{itemize}
\end{theorem}

\begin{example}[A counterexample]\label{ex:not-equivalence}
\rev{The additional regular-variation assumptions on $r_F$ or $f$ are not
consequences of $F\in\mathrm{DA}(\alpha)$ alone.  They are smoothness
assumptions on the tail. To illustrate this point, let us construct an example of a cdf $F \in \mathrm{DA}(\alpha)$, with a density $f$ not in $\RV_{-\alpha-1}$. Start from the Pareto density
$f_0(x)
=
\alpha x^{-\alpha-1}\mathbf 1_{[1,\infty)}(x) \, ,
$
and, for each integer $k\geq 1$, define
\[
m_k
:=
\int_k^{k+1} f_0(x)\,\dint x
=
k^{-\alpha}-(k+1)^{-\alpha}.
\]
We redistribute the mass $m_k$ inside the interval $[k,k+1)$ by setting
\[
f(x)=
\begin{cases}
\dfrac{3}{2}m_k,
& x\in \left[k,k+\dfrac12\right) \, ;\\
\dfrac12 m_k,
& x\in \left[k+\dfrac12,k+1\right)\, .
\end{cases}
\]
The function $f$ is a probability density on $[1,\infty)$. The function $f$ is not regularly varying, since
${f(k+\tfrac14)}/{f(k+\tfrac34)}=3$ for every $k\ge1$. However, for every integer $k\geq1$,
$\overline F(k)
=
\int_k^{+\infty} f(x)\,\dint x
=
\sum_{j\geq k}m_j
=
k^{-\alpha}.
$
Moreover, if \(x\in[k,k+1)\), then
\[
  (k+1)^{-\alpha}
  \le
  \overline F(x)
  \le
  k^{-\alpha},
\]
and therefore \(\overline F(x)\sim x^{-\alpha}\). Hence
\(\overline F\in RV_{-\alpha}\), so that \(F\in\mathrm{DA}(\alpha)\). 

Note that the density is piecewise constant and
therefore does not satisfy the continuity part of Assumption~0. A smoothed version, obtained by smoothing \(f\) in small neighbourhoods of the endpoints and midpoints of the intervals \([k,k+1)\), would lead to the same conclusion.
Note that $F$ does not satisfy the second-order regular variation assumptions either (see Appendix \ref{sn:not-SORV}).}
\end{example}

\begin{remark}[Relation with second-order regular variation]
\label{rem:second-order}
\rev{Classical rates of convergence for maxima are often obtained under
second-order regular variation assumptions. A standard formulation is that, for
some \(\rho\le0\) and some function \(A(t)\to0\) of constant sign,
\[
  \frac{\overline F(tx)/\overline F(t)-x^{-\alpha}}{A(t)}
  \longrightarrow
  x^{-\alpha}\frac{x^\rho-1}{\rho},
  \qquad x>0,
\]
with the usual interpretation \(x^{-\alpha}\log x\) when \(\rho=0\).
This condition describes the first non-vanishing correction to regular
variation of the integrated tail.

Our discrepancy is not written directly in terms of the survival ratio
\(\overline F(tx)/\overline F(t)\), but in terms of the reverse hazard rate
\(r_F=f/F\). Under additional smoothness assumptions, second-order regular
variation of the tail entails a corresponding second-order expansion of
\(r_F(tx)/r_F(t)\). In that case, the integrand in
\(\Delta(F_n\mid\Phi_\alpha)\), or in \(\Delta(\Phi_\alpha\mid F_n)\) depending
on the ordering of the supports, contains two contributions: a second-order
tail contribution of order \(|A(a_n)|\), and a finite-\(n\) contribution of
order \(1/n\). Thus, at the heuristic level made explicit in
Example~\ref{ex:second-order-benchmark}, our discrepancies become 
\[
  \Delta
  =
  O\bigl(|A(a_n)|+n^{-1}\bigr),
\]
provided the corresponding averaged remainder is integrable.}
\end{remark}

\begin{remark}
    \rev{Let us also comment on the comparison with more direct EVT
computations.  In dimension one, Kolmogorov and Wasserstein-1 distances
often admit explicit representations in terms of cdfs or quantile
functions.  Therefore, in examples where these objects are simple
enough, the rates can indeed be recovered by direct calculations.  This
is the case, for instance, for the Pareto example, and partly for the
other examples considered below.  The results given in Table \ref{tab:rest} illustrate what one can obtain with this ad-hoc approach, in the specific case of Kolmogorov distance. Similar results also hold for total variation and Wasserstein-, and many more.  The purpose  of our Stein approach is not
to replace these formulas whenever they are available, but to provide a
single comparison principle which yields bounds in several distances
from the same analytic object. }
\end{remark}

We now illustrate the bounds through several applications. We start with a
general calculation which explains how the Stein discrepancy should be
read under a smooth second-order regular variation assumption.  This
will serve as a benchmark for the examples that follow.

\begin{example}[A formal smooth second-order benchmark]
\label{ex:second-order-benchmark}
\rev{The form of the discrepancy depends on the ordering of the supports.
Most heavy-tailed distributions with positive lower endpoint fall under the case
\(c_F\ge0\), where the relevant discrepancy is
\(\Delta(F_n\mid\Phi_\alpha)\). Suppose, in this case, that
\(a_n=F^{\leftarrow}(1-1/n)\). Assume that the reverse hazard rate admits the
smooth second-order expansion
\[
  \frac{r_F(tx)}{r_F(t)}
  =
  x^{-\alpha-1}
  \{1+A(t)B(x)+o(A(t))\},
  \qquad x>0,
\]
uniformly on the range relevant for the integral, where \(A(t)\to0\) has
constant sign. Assume also that
\[
  \frac{t r_F(t)}{\overline F(t)/F(t)}
  =
  \alpha\{1+A(t)b+o(A(t))\}
\]
for some constant \(b\). Since \(a_n=F^{\leftarrow}(1-1/n)\), one has
\(\overline F(a_n)\sim 1/n\), and therefore
\[
  \frac{\alpha}{x^{\alpha+1}n a_n r_F(a_nx)}
  =
  1
  -
  A(a_n)\{B(x)+b\}
  -
  \frac{x^{-\alpha}}{n}
  +
  o\bigl(|A(a_n)|+n^{-1}\bigr),
\]
at the formal level and under the usual domination assumptions. Hence
\[
  1-
  \frac{\alpha}{x^{\alpha+1}n a_n r_F(a_nx)}
  =
  A(a_n)\{B(x)+b\}
  +
  \frac{x^{-\alpha}}{n}
  +
  o\bigl(|A(a_n)|+n^{-1}\bigr).
\]
Consequently,
\[
  \Delta(F_n\mid\Phi_\alpha)
  =
  O\bigl(|A(a_n)|+n^{-1}\bigr).
\]
The term \(|A(a_n)|\) is the smooth second-order tail contribution, while the
term \(1/n\) is the finite-sample contribution coming from the distribution of
the maximum.}

\rev{If \(c_F\le0\), the support ordering is reversed and one uses
\(\Delta(\Phi_\alpha\mid F_n)\). Under the analogous smooth expansion for
\(r_F\), the same principle holds, but the average is taken against the
Fr\'echet density \(\phi_\alpha(x)\,\dint x\) rather than against
\(f_n(x)\,\dint x\). The examples below illustrate both support
configurations.}
\end{example}

\begin{example}[Maxima of independent Pareto]
\label{ex:pareto2}
Let
\[
  F(x)=(1-x^{-\alpha})\mathbb{I}_{(1,\infty)}(x)
\]
be the Pareto cdf.  Then $a_n=n^{1/\alpha}$ and, with
$F_n(x)=F(n^{1/\alpha}x)^n$, the supports dictate that we apply
Theorem~\ref{mainresult} with $Q=F_n$ and $P=\Phi_\alpha$.  Immediate
computations yield
\[
   \Delta(F_n\mid\Phi_\alpha)
   =
   \frac1n
   \int_{n^{-1/\alpha}}^\infty x^{-\alpha}f_n(x)\,\dint x
   =
   \frac1{n+1}.
\]
This is the correct order of convergence in the Kolmogorov and total
variation distances for this simple problem.  Moreover, if
$\alpha>1$, then
\[
  \Delta_w(F_n\mid\Phi_\alpha)
  =
  \frac1n
  \int_{n^{-1/\alpha}}^\infty x^{1-\alpha}f_n(x)\,\dint x
  =
  \frac{n^{-1/\alpha}\Gamma(n+1)}{\Gamma(2+n-1/\alpha)}
  \Gamma\left(2-\frac1\alpha\right)
  \sim
  \frac{\Gamma(2-1/\alpha)}{n}.
\]
Hence the Wasserstein-1 bound is also of order $1/n$.

\rev{This is precisely the benchmark situation in which the tail
approximation itself has no second-order error: the Pareto tail is
exactly regularly varying.  Thus the term $A(a_n)$ in
Example~\ref{ex:second-order-benchmark} is absent, and the whole
contribution comes from the finite-sample term $1/n$.  The log-logistic
cdf
\[
  F(x)=\frac{x^\alpha}{1+x^\alpha}\mathbb{I}_{(0,\infty)}(x)
\]
leads to the same order.  In that case the first correction to the tail
is of order $x^{-\alpha}$, which under the normalization
$a_n=n^{1/\alpha}$ is of order $1/n$, the same as the finite-sample
contribution.}
\end{example}

\rev{
\begin{example}[Maxima of independent Cauchy]
\label{sec:cauchy}
Let
\[
  X_n:=\frac{\pi}{n}\max(Y_1,\ldots,Y_n),
\]
where \(Y_1,\ldots,Y_n\) are i.i.d.\ standard Cauchy random variables,
with density \(f(x)=(\pi(1+x^2))^{-1}\).  The Cauchy cdf belongs to
\(\mathrm{DA}(1)\).  Since its support is the whole real line, the
Fr\'echet distribution has the smaller support, and we use
\(\Delta(\Phi_1\mid F_n)\). 

{This example falls under the second support configuration described
in Example~\ref{ex:second-order-benchmark}.  The Cauchy tail satisfies
\[
  \overline F(x)
  =
  \frac1{\pi x}+O(x^{-3}),
  \qquad x\to\infty.
\]
Thus the first non-vanishing tail correction is of order \(x^{-2}\).
With the normalization \(a_n=n/\pi\), the second-order contribution is
of order \(a_n^{-2}=O(n^{-2})\), whereas the finite-sample contribution
is of order \(1/n\).  Hence the leading order is expected to be \(1/n\),
as in the Pareto benchmark.}

This is confirmed by the direct computation in Appendix~ \ref{sn:appendccaychy}: there exists a
constant \(C_1\) such that
\[
  \Delta(\Phi_1\mid F_n)\le \frac{C_1}{n}.
\]
Numerical evaluations indicate that \(C_1=25\) is admissible.

Appendix~ \ref{sn:appendccaychy} also shows that the weighted discrepancy satisfies
\[
  \Delta_w(\Phi_1\mid F_n)\le \frac{C_2}{n},
\]
with \(C_2=2\) admissible. This last estimate should be understood only as a weighted discrepancy estimate. Since \(\Phi_1\) and the normalized Cauchy maximum do not have finite first moments, it does not yield a Wasserstein-1 
bound.

\end{example}}

\begin{example}[Maxima of independent Burr XII]
\label{sec:burr}
Let
\[
  X_n=n^{-1/(\alpha\tau)}\max(Y_1,\ldots,Y_n),
\]
where $Y_1,\ldots,Y_n$ are i.i.d.\ random variables with cdf
\[
  F(x)=1-(1+x^\alpha)^{-\tau},
  \qquad x>0,\quad \alpha>0,\quad \tau>1.
\]
This distribution belongs to $\mathrm{DA}(\alpha\tau)$.  Here
$c_F=0$, so the supports of $F_n$ and $\Phi_{\alpha\tau}$ have the same
left endpoint.  We use the discrepancy $\Delta(\Phi_{\alpha\tau}\mid
F_n)$.

\rev{The Burr XII example displays the opposite regime from the Pareto
and Cauchy examples: the second-order tail contribution is visible at
the leading order.  Indeed,
\[
  \overline F(x)
  =
  (1+x^\alpha)^{-\tau}
  =
  x^{-\alpha\tau}
  \left(1-\tau x^{-\alpha}+O(x^{-2\alpha})\right),
  \qquad x\to\infty.
\]
With $a_n=n^{1/(\alpha\tau)}$, this correction has order
$a_n^{-\alpha}=n^{-1/\tau}$.  Since $\tau>1$, it dominates the
finite-sample contribution $1/n$.  The Stein discrepancy is therefore
expected to be of order $n^{-1/\tau}$.} 

This is confirmed by a direct computation provided in Appendix \ref{sn:exampburr} where it is shown that 
\[
  \Delta(\Phi_{\alpha\tau}\mid F_n)
  \sim
  (\tau+1)
  \mathbb{E}\bigl[E_n^{1/\tau}\bigr]
  =
  (\tau+1)
  \Gamma\left(1+\frac1\tau\right)n^{-1/\tau}.
\]
  Numerical evaluations indicate that a similar scaling holds
for $\Delta_w$, but the resulting limit depends in a nontrivial manner
on $\alpha$; see Figure~\ref{fig:byrr}. The numerical evaluations can be found in the  supplementary material \citep{mansanonline}. 
\end{example}

\begin{figure}
  \centering
  \includegraphics[scale=.6]{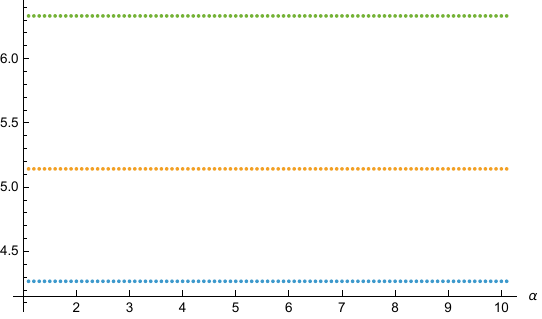}
  \qquad\qquad
  \includegraphics[scale=.6]{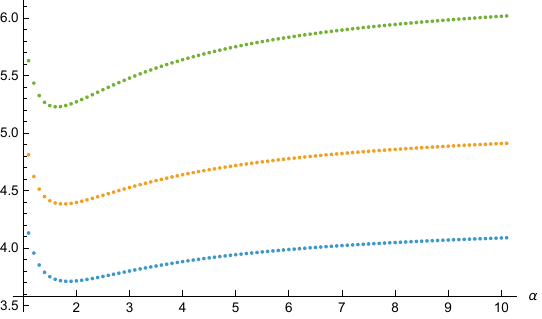}
  \caption{\it Numerical evaluation (with \texttt{Mathematica}) of
    $n^{1/\tau}\Delta(\Phi_{\alpha\tau}\mid F_n)$ (\rev{left} plot) and
    $n^{1/\tau}\Delta_w(\Phi_{\alpha\tau}\mid F_n)$ (\rev{right} plot) for
    $n=10^5$, as a function of $\alpha\in[1,10]$, for $\tau=3$
    (blue curve), $\tau=3.5$ (orange curve), and $\tau=4$ (green
    curve), when $F$ is the Burr XII distribution.}
  \label{fig:byrr}
\end{figure}

\begin{example}[A counterexample, continued]\label{ex:acounter}
Let \(F\) be the oscillating-density distribution from Example~2.6.
Take the usual normalization \(a_n=n^{1/\alpha}\). Although this
piecewise-constant example does not satisfy the continuity part of
Assumption~0, the discrepancy can still be computed explicitly. The
calculation in Appendix~B.4 shows that
\[
  \Delta(F_n\mid \Phi_\alpha)\longrightarrow \frac12.
\]
Thus the reverse-hazard Stein discrepancy is stronger than mere membership
in \(\mathrm{DA}(\alpha)\). The normalized maxima still converge in
distribution to \(\Phi_\alpha\), but the local oscillations of the density
prevent the reverse hazard ratio from stabilizing in the averaged sense
measured by \(\Delta(F_n\mid\Phi_\alpha)\). The same phenomenon persists
after smoothing the density locally, provided the oscillation pattern is
preserved at infinity.
\end{example}
 
      Finally, we conclude with two proofs-of-concept examples which illustrate how the discrepancy can be used for inference purposes. 
      
\begin{example}[A bootstrap diagnostic based on the estimated reverse hazard]
\label{ex:bootstrap-diagnostic}
The counterexample above suggests a goodness-of-fit diagnostic for smooth
Pareto-like tail behaviour.

Given observations \(X_1,\ldots,X_N\), we first estimate the tail index by
the reciprocal Hill estimator \(\widehat\alpha\). We then form block maxima
\(M_1,\ldots,M_B\), where \(B=\lfloor N/m\rfloor\), and let
\(\widehat a_m\) be the empirical \((1-1/m)\)-quantile. Writing
\(Z_j=M_j/\widehat a_m\), define
\[
T_{N,m}
=
\frac1B
\sum_{j=1}^B
\left|
1-
\frac{\widehat\alpha}
{Z_j^{\widehat\alpha+1}m\widehat a_m
\widehat r(\widehat a_m Z_j)}
\right|,
\]
where \(\widehat r\) estimates the reverse hazard \(r_F=f/F\). In the
simulation below we estimate \(r_F\) on the logarithmic scale. More precisely,
if \(g\) denotes the density of \(\log X\), we use
\[
\widehat r(x)
=
\frac{\widehat g_h(\log x)}{x\widehat F_N(x)},
\]
where \(\widehat g_h\) is a kernel density estimator and \(\widehat F_N\) is
the empirical distribution function.

The statistic is calibrated by a parametric bootstrap under the fitted Pareto
tail. For each bootstrap sample we recompute the whole procedure, including
\(\widehat\alpha^\ast\), \(\widehat a_m^\ast\), \(\widehat r^\ast\), and
\(T_{N,m}^\ast\). With \(R_{\mathrm{boot}}\) bootstrap samples, the
bootstrap \(p\)-value is
\[
p
=
\frac{
1+\#\{b:T_{N,m}^{\ast,b}\geq T_{N,m}^{\mathrm{obs}}\}
}{
R_{\mathrm{boot}}+1
}.
\]

We compare a smooth Pareto law with the oscillating-density distribution from
Example~\ref{ex:not-equivalence}. Both distributions have tail index
\(\alpha=2.5\). The procedure is not given the form of the oscillating
alternative; it only uses the estimated reverse hazard. The simulation uses
\(N=100\,000\), block size \(m=500\), Hill threshold \(k=2000\), bandwidth
\(h=0.005\), \(99\) bootstrap samples, and \(20\) Monte Carlo repetitions.

\[
\begin{array}{c|ccccccc}
\text{model}
&
\text{mean }\widehat\alpha
&
\text{sd }\widehat\alpha
&
\text{mean }T_{N,m}
&
\text{sd }T_{N,m}
&
\text{mean }q_{0.95}^{\ast}
&
\text{median }p
&
\text{rej. }5\%
\\
\hline
\text{Pareto}
& 2.4996 & 0.0469
& 0.3128 & 0.0232
& 0.3561
& 0.54
& 0.05
\\
\text{oscillating}
& 2.4241 & 0.0477
& 0.4197 & 0.0309
& 0.3599
& 0.01
& 0.95
\end{array}
\]

The statistic is not expected to be close to zero under the Pareto null,
because the reverse hazard is itself estimated nonparametrically. The bootstrap
calibration accounts for this estimation error. Under the smooth Pareto model
the rejection frequency is close to the nominal level, whereas the
oscillating-density model is rejected in most repetitions. This supports the
interpretation of the Stein discrepancy as a diagnostic for local
reverse-hazard regularity, rather than as a test of the tail index alone. A
systematic study of the tuning parameters \(m\), \(h\), and of the estimator
\(\widehat r\), is beyond the scope of the present paper.
\end{example}

\begin{example}[A Stein-corrected Hill estimator]
\label{ex:estimation}
The empirical Stein discrepancy can also be used to
post-process a preliminary tail-index estimate. Let \(\widehat\alpha_H\) be
the reciprocal Hill estimator, computed from the largest \(k\) observations. For a block
size \(m\), empirical normalizer \(\widehat a_m\), and normalized block maxima
\(Z_j=M_j/\widehat a_m\), define
\[
\widehat\Delta_{m,h}(\alpha)
=
\frac1B
\sum_{j=1}^B
\left|
1-
\frac{\alpha}
{Z_j^{\alpha+1}m\widehat a_m\widehat r_h(\widehat a_m Z_j)}
\right|,
\]
where \(\widehat r_h\) estimates \(r_F=f/F\), again by a kernel estimator on
the logarithmic scale. We then set
\[
\widehat\alpha_{\mathrm{SC}}
=
\arg\min_{\alpha\in[\widehat\alpha_H-\eta,\widehat\alpha_H+\eta]}
\widehat\Delta_{m,h}(\alpha).
\]
Thus Hill provides the stable first-order estimate, while the Stein discrepancy
uses the estimated reverse hazard to apply a local correction. 

As an illustration, consider
\[
\overline F(x)=\frac{e^\alpha x^{-\alpha}}{\log x},
\qquad x\ge e,
\]
with \(\alpha=2.5\). This model has a slowly varying correction, and therefore
induces an upward finite-threshold bias for Hill. We simulated \(30\) samples
of size \(N=10^6\), used \(k=20\,000\), \(m=5000\), and \(\eta=0.3\). The
following table compares Hill with the Stein-corrected estimator for several
log-scale bandwidths \(h\).

\[
\begin{array}{c|cccc}
\text{estimator}
& \text{mean} & \text{sd} & \text{bias} & \text{rmse}
\\
\hline
\text{Hill}
& 2.8938 & 0.0233 & 0.3938 & 0.3944
\\
\text{Stein-corrected},\ h=0.003
& 2.8024 & 0.1110 & 0.3024 & 0.3215
\\
\text{Stein-corrected},\ h=0.005
& 2.7938 & 0.1262 & 0.2938 & 0.3189
\\
\text{Stein-corrected},\ h=0.010
& 2.7938 & 0.1289 & 0.2938 & 0.3199
\\
\text{Stein-corrected},\ h=0.020
& 2.8058 & 0.1239 & 0.3058 & 0.3291
\\
\text{Stein-corrected},\ h=0.050
& 2.8058 & 0.1250 & 0.3058 & 0.3295
\end{array}
\]

In this example, the Stein correction reduces the slow-variation bias of Hill
and improves RMSE for a suitable range of bandwidths, at the cost of increased
variance. The example therefore suggests a possible bias--variance tradeoff
driven by the estimated reverse-hazard structure. A systematic study of the
tuning parameters \(m\), \(h\), \(k\), and \(\eta\) is left for future work.
\end{example}

\section{Stein's method and a proof of Theorems \ref{mainresult} and
  \ref{mainresultwass}}
\label{sec:steins-meth-comp}

Let $P$ be a cdf on $\mathbb{R}$ satisfying Assumption 0, having no
atoms (i.e. $p_0 = 0$ in \eqref{eq:4}). We follow the blueprint from
\cite{ley2017stein} to set up ``a Stein's method for $P$'', in three
steps.

\medskip

\subsection{A Stein operator for $P$}
 The \emph{canonical Stein operator} for $P$ is the linear
operator
$$ \varphi \mapsto \TT_{p}\varphi := {(p \varphi)'}/{p}\I_{(c_P,
  +\infty)} $$ acting on the set of functions
$\varphi: \mathbb{R} \to \mathbb{R}$ such that $p \varphi $ is
differentiable on $(c_P, +\infty)$.  The \emph{Stein class} for
$\mathcal{T}_p$ is the collection $\FF(\mathcal{T}_{p})$ of functions $\varphi: \mathbb{R} \to \mathbb{R}$ such that
\begin{enumerate}[(i)]
        \item  $p\varphi$ is
differentiable on $(c_P, +\infty)$, 
\item $(p\varphi)'$ is integrable
on $(c_P, +\infty)$, and 
\item 
$\lim_{x \to c_P^+}p(x)\varphi(x) = \lim_{x \to
  +\infty}p(x)\varphi(x)$.  
\end{enumerate}
The pair
$(\mathcal{F}(\mathcal{T}_p), \mathcal{T}_p)$ is the \emph{canonical
  Stein pair} for $P$.  
  
  Conditions (i) and (ii) are equivalent to
requiring that $(p \varphi)$ is absolutely continuous on every compact
subinterval of $(c_P, +\infty)$, which we denote
$(p\varphi) \in \mathrm{AC}_{\mathrm{loc}}((c_P, +\infty))$. In
particular these conditions guarantee that $P(\TT_{p}\varphi) =0$ for
all $\varphi \in \FF(\mathcal{T}_p)$.  This last fact characterizes $P$,
under certain additional conditions (see e.g.\ \cite{ley2017stein}).

Typically $(\mathcal{F}(\mathcal{T}_p), \mathcal{T}_p)$ is not
tractable, and it is best to work with a modified operator
$\mathcal{A}_p\cdot :=\TT_{p}\left(m\,\cdot \right)$ for $m$ a
well-chosen function. Such an operator is called a
\emph{reparameterization}.  In the present paper we choose $m = P/p$,
so that the reparameterized operator we use is
\begin{align}\label{eq:8}
  \mathcal{A}_{p} \varphi & := \TT_{p}\left(\frac{ P}{p}\, \varphi
                            \right) = \frac{(  P \varphi)'}{p}\I_{(c_P, +\infty)}  =
                            \left(\frac{ P}{p} \varphi' + \varphi \right)\I_{(c_P, +\infty)}.
\end{align}
We associate to $ \mathcal{A}_{p}$ the reparameterized Stein class
$\mathcal{F}(\mathcal{A}_p)$, which collects all
$\varphi : \mathbb{R} \to \mathbb{R}$ such that
$(P \varphi) \in \mathcal{F}(\mathcal{T}_p)$; in other words
$\varphi \in \mathrm{AC}_{\mathrm{loc}}((c_P, +\infty))$ and
$\lim_{x \to c_P^+} P(x) \varphi (x) = \lim_{x \to +\infty} \varphi
(x)$. By construction $P(\mathcal{A}_p \varphi) = 0$ for all
$\varphi \in \mathcal{F}(\mathcal{A}_p)$.

\medskip

\subsection{The Stein equation
    for $P$, its solution, and bounds thereon} 
    Given $h \in L^1(P)$,
the $h$-Stein equation is the ODE on $(c_P, +\infty)$:
\begin{align}\label{eq:SteinID}
  \frac{P}{p} \varphi'+  \varphi  =  h - P(h).
\end{align}
The unique solution to \eqref{eq:SteinID} which belongs to
$\mathcal{F}(\mathcal{A}_p)$ is defined on
$(c_P, +\infty)$ as
\begin{align}
f_{h}(x) & =  \frac{1}{P(x)}\int_{c_P}^{x}
             (h(t)-P(h) )p(t)\, \dint t  = \frac{1}{P(x)} \int_x^{+\infty}
           (P(h) - h(t)  )p(t)\, \dint t;   \label{eq:5}
\end{align}
we set $f_h$ to 0 for all $x \le c_P$.  The relevant properties of the
solutions $f_{h}$ are provided in the next proposition.

\begin{proposition}\label{genSteinsol} Let \( P \) be an atomless
  cumulative distribution function (cdf) satisfying Assumption 0, and
  let $f_h$ be defined by \eqref{eq:5} with $h \in L^1(P)$.  Then
  $f_h$ is continuous and bounded on $(c_P, +\infty)$ with
  \begin{equation}
    \label{eq:6}
    \lim_{x \to c_P^+} f_h(x) = h(c_P) - P(h) \mbox{ and }  \lim_{x \to
      c_P^+} P(x) f_h(x) =   \lim_{x \to
      +\infty} f_h(x) = 0.
  \end{equation}
  For all \( x > c_P \), the following holds:
\begin{enumerate}
    \item \label{item:1} If \( h \in \mathcal{H}_{\mathrm{TV}} \), then
    \[
    |f_h(x)| \leq 1
    \quad \text{and} \quad
    \frac{ P(x)}{p(x)}|f_h'(x)| \leq 2.
    \]

    \item \label{item:2} If \( h \in \mathcal{H}_{\mathrm{Kol}} \), then
    \[
    |f_h(x)| \leq 1
    \quad \text{and} \quad
    \frac{ P(x)}{p(x)}|f_h'(x)|  \leq 1.
    \]

  \item \label{item:3} Suppose \( P \in L^1 \), let \( \mathrm{id}(x)
    = x \), \rev{write $\overline P(x)=1-P(x)$,} and consider
  \begin{align}
    \label{eq:me1e2}
   & m(x) =P(|x - \mathrm{id}|), \qquad  e_1(x) = \frac{P(x) \wedge
      \overline{P}(x)}{P(x)}, \mbox{ and } \nonumber \\
   & e_2(x) = \frac{\left( \int_{c_P}^{x} P(t)\, \dint t \right) \wedge
      \left( \int_{x}^{+\infty} \overline{P}(t)\, \dint t \right)}{P(x)}.
  \end{align}
       If \( h \in \mathrm{Lip}(1) \), then
    \[
    |f_h(x)| \leq m(x)\, e_1(x) + e_2(x),
    \qquad
    \frac{ P(x)}{p(x)}|f_h'(x)|\leq m(x)(1+e_1(x)) + e_2(x).
    \]
\end{enumerate}
\end{proposition}

\begin{proof}
  Continuity of $f_h$ follows by definition. To see that the function
  is bounded, we only need to check the behavior at the edges of the
  support. The second claim in \eqref{eq:6} is immediate; the first
  claim follows from L'Hospital's rule which applies since $P(c_P)= 0$
  by assumption.

\medskip

We now proceed to prove the three claims. The first claim in Item
\ref{item:1} is immediate:
\begin{align*}
  |f_{h}(x)| & \leq \frac{1}{ P(x)} \int_{c_P}^{x}
               \bigl|h(t)-P(h)  \bigr |p(t)\,\dint t   \leq
               \frac{1}{P(x)}\bigl
               \|h- P(h)\bigr \|_{\infty}\int_{c_P}^{x}
               p(t)\, \dint t,
\end{align*}
and $ \sup_{h \in \mathcal{H}}\bigl \|h-P(h)\bigr \|_{\infty}  = \max
\left( \mathbb{P}[B], 1-\mathbb{P}[B], B \in \mathcal{B} \right) =   1$ for
$h \in \mathcal{H}_{TV}$. Next, using the Stein equation
\eqref{eq:SteinID}, one gets
\begin{align*}
  \frac{P(x)}{p(x)}|f'_{h}(x)| & \le   |h(x)-P(h)- f_{h}(x)
                                 \bigr | \le 2  \| h-P(h)
                                 \|_{\infty}.
\end{align*}
 Item \ref{item:1} is proved.

\medskip

In the case of Item \ref{item:2}, we have more refined knowledge of
the solution and thus we can obtain an improvement on the previous
claim, as follows. Suppose that $z > c_P$.  From
\cite{Ernst2022,Ernst2020} we know that if
$h(x) = \mathbb{I}[x \le z]$ then
\begin{align*}
  f_{\mathbb{I}[\cdot \le z]}(x) = \frac{P(x \wedge z)
  \overline{P}(x \lor z)}{P(x)} =
  \begin{cases}
\overline  P(z)  & \mbox{ if } x \le z\\
 P(z)        \frac{\overline P(x)}{P(x)} \ & \mbox{ if } x > z
\end{cases}.
\end{align*}
Since the function ${\overline{P}}/{P}$ is
decreasing, we can conclude
$ |f_{\mathbb{I}[\cdot \le z]}(x)| \leq \overline{P}(z)$ which yields
a non-uniform bound, and hence the first part of the claim.  Next, we
know
\begin{align*}
  \frac{P(x)}{p(x)}  f'_{\mathbb{I}[\cdot \le z]}(x) =
  \mathbb{I}[x \le z]-P(z) -  \frac{P(x \wedge z)
  \overline{P}(x \lor z)}{P(x)}.
\end{align*}
If $x < z$ then
\begin{align*}
  \frac{P(x)}{p(x)}     f'_{\I[\cdot \leq z]}(x) = 0,
\end{align*}
and, on the other hand, if $x>z$ then
\begin{align*}
  \frac{P(x)}{p(x)}    f'_{\I[\cdot \leq z]}(x)  & = 0 - P(z) -
                                                   P(z)        \frac{\overline P(x)}{P(x)}= -{P}(z)
                                                   \frac{1}{P(x)}.
\end{align*}
Taking absolute values and using decreasingness of $1/P$, the claim
follows.

\medskip

We now tackle Item \ref{item:3}. First we note that
$\lim_{x \to + \infty} x \overline P(x) = 0$ because $P$ is
integrable. Since $h \in \mathrm{Lip}(1)$, integration by parts is
allowed and yields (still with $x > c_P$)
\begin{eqnarray*}
 && \int_{c_P}^{x} (h(t)-P(h))p(t)\,{d}t  = \int_{c_P}^{x} h(t)p(t)\, \dint t -
                                          P(h) P(x)   \\
                                          && =
                                          h(x)   P(x) - \int_{c_P}^{x}
                                          h'(t)P(t) \, \dint t - P(h) P(x)  = P(h(x) -h) P(x)- \int_{c_P}^{x} h'(t)P(t) \, \dint t.
\end{eqnarray*}
Since $||h'||_{\infty} \le 1$ and $|h(x) - h(t) | \le |x - t|$, we
deduce
\begin{align*}
  \left |\int_{c_P}^{x} (h(t)-P(h))p(t)\, \dint t \right | & \leq  P(|h(x)
                                                     -h|)P(x) +
                                                     \int_{c_P}^{x}
                                                     \left |h'(t)
                                                     \right |
                                                           P(t) \,
                                                     \dint  t \\
                                                   & \leq
                                                     P(|x-\mathrm{id}|)
                                                     P(x)
                                                     +  \int_{c_P}^{x}P(t) \, \dint  t\rev{.}
\end{align*}
Similarly,
\begin{align*}
  \int_{x}^{+\infty} (h(t)-P(h))p(t)\,\dint t  = P(h(x) -h) \overline
  P(x) + \int_{x}^{+\infty} h'(t) \overline P(t) \, \dint t.
\end{align*}
and thus
\begin{align*}
  \left|  \int_{x}^{+\infty} (h(t)-P(h))p(t)\,\dint t    \right| \le P(|x-\mathrm{id}|)
  \overline P(x)
  +
  \int_{x}^{+\infty}\overline P(t) \, \dint  t.
\end{align*}
Dividing by $P(x)$, the result follows.  The claim concerning the
derivative follows from the Stein equation, exactly as for the
previous two items.
\end{proof}

\subsection{From $P$ to $Q$, the transfer principle}
Let $Q$ be another cdf on $\mathbb{R}$. Suppose that $Q$ satisfies
Assumption 0, with support having lower endpoint $c_Q \ge c_P$. Let
$\HH \in L^1(P) \cap L^1(Q)$ be a class of functions on $\R$.  As
outlined in the introduction, we aim to bound probability distances
induced by $\HH$ through
\begin{align*}
  \mathcal{D}(P, Q \mid \mathcal{H}) & = {\sup}_{{h\in \HH}} \, \bigl |(Q-P) (h) \bigr |.
\end{align*}
Obviously, $(Q-P) (h) = Q(h - P(h))$. Since the support of $Q$ is a
subset of the support of $P$, we can use \eqref{eq:SteinID} then
\eqref{eq:8} to reap
\begin{align*}
  \mathcal{D}(P, Q \mid \mathcal{H})  & ={\sup}_{{h\in \HH}} \, \bigl
                                        | Q (h - P(h))
               \bigr | =  {\sup}_{{h\in \HH}} \, \left |Q \left
               (\frac{P}{p}f'_{h}+f_{h} \right ) \right |.
\end{align*}
Next define a Stein operator for $Q$ through
\begin{align*}
  \mathcal{A}_{q} \varphi &  :=
                            \left(\frac{Q}{q} \varphi'+ \varphi \right)\I_{(c_Q, +\infty)}.
\end{align*}
Then, since $f_h$ is bounded, and from \eqref{eq:6}, satisfies
$\lim_{x \to + \infty} f_h(x) Q(x) = 0$ and
$\lim_{x \to c_Q^+} f_h(x) Q(x) = f_h(c_Q) q_0$ is finite, we get
\begin{align*}
  Q \bigl ( \mathcal{A}_q f_h \bigr ) = -  f_h(c_Q)  q_0.
\end{align*}
It follows that
\begin{align}
  \mathcal{D}(P, Q \mid \mathcal{H})   & = {\sup}_{{h\in \HH}} \, \left| Q \left
                                         ( \mathcal{A}_pf_h \right )  \right|
                                         =  {\sup}_{{h\in \HH}} \, \left| Q \left
                                         ( \mathcal{A}_pf_h
                                         \right )  - Q  \left(
                                         \mathcal{A}_q f_h   \right)
                                        +  f_h(c_Q)  q_0 \right| \nonumber \\
                                       & =  {\sup}_{{h\in \HH}} \, \left| Q \left
                                         (\left( \frac{P}{p} - \frac{Q}{q}\right) f'_{h} \right
                                         )  +  f_h(c_Q)  q_0  \right|
                                         \nonumber \\
                                       & =  {\sup}_{{h\in \HH}} \, \left| Q \left
                                         (\left( 1 - \frac{Q}{q}\frac{p}{P}\right) \frac{P}{p}f'_{h} \right
                                         )  +  f_h(c_Q)  q_0 \right|. \label{eq:12}
\end{align}
Taking suprema and using the various relevant statements from
Proposition \ref{genSteinsol} yields
Theorem~\ref{mainresult}. Similarly, we get the claim from
Theorem~\ref{mainresultwass} upon noting that if $c_P \ge 0$ then
$m(x) (1+e_1(x)) + e_2(x) \le 3 x + 2 P(\mathrm{id})$.

\section{Proof of Theorem \ref{borneXalpha}}
\label{sec:further-proofs}

We first consider a cdf $F$ satisfying Assumption~0, with support
$\mathrm{supp}(F)=(c_F,+\infty)$ and $c_F\le0$. \rev{We assume that
$F\in\mathrm{DA}(\alpha)$, that $r_F\in\mathrm{RV}_{-\alpha-1}$, and
that $r_F$ is bounded on finite intervals of $\R_{\geq 0}$.}
We want to show that $\Delta(\Phi_\alpha\mid F_n)$ goes to $0$ as $n$
grows large. For every $n$ in $\N^*$, set $u_n := { n a_n r_F(a_n)}/{\alpha}$. Then for
all $x \ge c_F/a_n$:
\begin{align*}
 1- \frac{x^{1+\alpha}}{\alpha} n a_n r_F(a_nx)
  & =1- x^{1+\alpha}\frac{r_F(a_nx)}{r_F(a_n)}u_n
                                                                                                                   =
       1-                                                                                                            u_n
    - u_nx^{1+\alpha} \left ( \frac{r_F(a_nx)}{r_F(a_n)}-
    \frac{1}{x^{1+\alpha}} \right ) \\
  & =: 1-u_n-u_nx^{1+\alpha}g_F(x,a_n).
\end{align*}
 It follows that
\begin{align*}
  \left|  1- \frac{x^{1+\alpha}}{\alpha} n a_n r_F(a_nx)  \right|
  &
    \leq |1-u_n| + |u_n|\left |x^{1+\alpha}g_F(x,a_n) \right | \mbox{
    on }(0, +\infty).
\end{align*}
We tackle these two terms separately.

We first deal with $|1-u_n|$. \rev{Since
$r_F\in\mathrm{RV}_{-\alpha-1}$ and $\overline F\in\mathrm{RV}_{-\alpha}$,
Karamata's theorem gives
\[
  \frac{a_n r_F(a_n)}{\overline F(a_n)/F(a_n)}
  = \frac{a_n f(a_n)}{\overline F(a_n)} \longrightarrow \alpha,
\]
where we used $F(a_n)\to1$.  Together with
$n\overline F(a_n)\to1$, this implies $u_n\to1$.}

Next we deal with the second term. By \rev{assumption}, $r_F$
belongs to $\RV_{-\alpha-1}$. Since $r_F$ is bounded on finite
intervals of $\R_{\geq 0}$, by Proposition \ref{gammaleq0} applied to $r_F$,
with $\delta = 1/2$, and $\xi = 1$, there exists $t_0 >0$ such that
for every $t\geq t_0$,
\begin{align*}
    \frac{r_F(tx)}{r_F(t)} \leq Mx^{-1/2-1-\alpha}
\end{align*}
for every $0<x\leq 1$, where $M>0$ is a constant. Hence
\begin{align} \label{part1aless1}
    x^{1+\alpha}|g_F(x,t)| & = x^{1+\alpha} \left | \frac{r_F(tx)}{r_F(t)}
                             - \frac{1}{x^{1+\alpha}} \right |
                             \leq 1 + x^{1+\alpha}\frac{r_F(tx)}{r_F(t)}
                             \leq 1 + Mx^{-1/2}
\end{align}
for every $x$ in $]0,1]$. Furthermore, using Proposition \ref{Drees}
for $\varepsilon =1$ and $\delta = \alpha/2$, one gets that there
exists $t_1 >0$ such that
\begin{align}\label{auxaux}
    |g_F(x,t)| = \left | \frac{r_F(tx)}{r_F(t)} - \frac{1}{x^{1+\alpha}} \right |\leq \frac{1}{x^{1+\alpha/2}}
\end{align}
for $x \geq t_1/t$, and $t>0$. It follows that
\begin{align} \label{part1gret1}
|x^{1+\alpha}g_F(x,t)| & \leq \left ( 1 + \frac{M}{x^{1/2}} \right )
                         \I_{_{[0,1]}}(x) +
                         x^{\alpha/2}\I_{_{[1,+\infty[}}(x) =: \Psi_{\alpha}(x)
\end{align}
for $t$ big enough and thus, in particular,
$|x^{1+\alpha}g_F(x,a_n)| \le \Psi_{\alpha}(x)$ for all $n$ large
enough.  Clearly $\Psi_{\alpha} \in L^1(\phi_\alpha)$ for all
$\alpha>0$. Hence the conclusion holds by Lebesgue's dominated
convergence theorem.

\bigskip

We now tackle the second item of the theorem: we take
$F\in\mathrm{DA}(\alpha)$ with support $(c_F,+\infty)$ and $c_F\ge0$,
\rev{and we assume in addition that $f\in\mathrm{RV}_{-\alpha-1}$ and
that $1/r_F$ is bounded on finite intervals of $\R_{\geq 0}$.}
Our aim is to prove that $\Delta(F_n\mid\Phi_\alpha)\to0$ as
$n\to+\infty$.  We begin by breaking the integral into
\begin{align*}
  &\Delta(F_n \mid \Phi_{\alpha})   \\ & = \int_{c_F/a_n}^1
                                      \left| 1 - \frac{\alpha}{x^{\alpha+1}} \frac{1}{n a_n r_F(a_nx)} \right|
                                      f_n(x)\, \dint x  + \int_{1}^{+\infty}
                                      \left| 1 - \frac{\alpha}{x^{\alpha+1}} \frac{1}{n a_n r_F(a_nx)} \right|
                                      f_n(x)\, \dint x
\end{align*}
and dealing with each integral separately.

For the first integral, we introduce the notations
\begin{equation} \label{fonctionsymp}
    q_n(x) := \left(1 - \frac{\alpha}{x^{\alpha+1}} \frac{1}{n a_n r_F(a_nx)}\right)
x^{1+\alpha}
\end{equation}
and $U(t) = F^{\leftarrow}(1-1/t)$.
Then, with the change of variables $x = U ( {n}/{v})$, it holds that
\begin{align*}
  &   \int_{c_F/a_n}^1
    \left| 1 - \frac{\alpha}{x^{\alpha+1}} \frac{1}{n a_n r_F(a_nx)} \right|
    f_n(x)\, \dint x  \\
    & = \int_{c_F}^{a_n}
    \left| 1 - \frac{\alpha a_n^{\alpha+1}}{x^{\alpha+1}} \frac{1}{n a_n r_F(x)} \right|
    n F(x)^{n-1}f(x)\, \dint x  \\
  &  = \int_{c_F}^{a_n}
    \frac{a_n^{1+\alpha}}{x^{1+\alpha}}\left |q_n \left ( \frac{x}{a_n}
    \right ) \right |n F(x)^{n-1} f(x) \, \dint x  \\
  & = \int_{1}^{+\infty} \left (\dfrac{a_n}{U \left ( \frac{n}{v}
    \right )}\right )^{1+\alpha} \left |q_n \left ( \frac{U \left (
    \frac{n}{v} \right )}{a_n} \right ) \right |\left ( 1 -
    \frac{v}{n} \right )^{n-1}\I[1\leq v \leq n] \, \dint v.
\end{align*}
Notice that since $1-F$ is in
$\RV_{-\alpha}$ and decreasing, the function $U$ belongs in
$\RV_{\frac{1}{\alpha}}$. Since $U(n) = a_n$, one gets
\begin{equation*}
    \frac{U\left ( \frac{n}{v} \right )}{a_n} \underset{n\rightarrow +\infty }{\longrightarrow} \left (\frac{1}{v} \right )^{\frac{1}{\alpha}}
\end{equation*}
for every $v$ in $[1,+\infty[$, and since $\frac{1}{\alpha}$ is
positive and $U$ is bounded near $1$, the convergence is
uniform. Besides, one has
\begin{equation*}
    \frac{c_F}{a_n} \leq \frac{U\left ( \frac{n}{v} \right )}{a_n} \leq 1.
  \end{equation*}
Furthermore, one has $\frac{1}{U}$ in $\RV_{-\frac{1}{\alpha}}$ and by
Proposition \ref{gammaleq0}, for every $\delta>0$, there exists
$M_\delta>0$ such that, for $n$ big enough,
\begin{equation*}
\frac{a_n}{U\left ( \frac{n}{v} \right )} \leq M_\delta \left (
  \frac{1}{v} \right )^{-\frac{1}{\alpha}-\delta} = M_\delta \,
v^{\frac{1}{\alpha}+\delta}
\end{equation*}
for every $v$ in $[1,+\infty[$.
Now, the sequence of positive functions
\begin{equation*}
  w_n(v) := \left ( \dfrac{a_n}{U\left ( \dfrac{n}{v} \right )} \right )^{1+\alpha}\left ( 1 - \frac{v}{n} \right )^{n-1}\I[1\leq v \leq n]
\end{equation*}
is pointwise convergent, and dominated, for $n$ large enough, by the integrable function
\[w(v) := M_\delta \,e^{-v} v^{(1+\alpha)(\frac{1}{\alpha}+\delta)}.\]
We can write
\begin{equation*}
    q_n(x) = x^{1+\alpha}-v_n L_n(x)
\end{equation*}
where
\begin{align*}
  L_n(x) := \frac{F(a_nx)}{f(a_nx)} \frac{f(a_n)}{F(a_n)} \mbox{ and
  }v_n := \frac{F(a_n)\alpha
  \overline{F}(a_n)}{n\overline{F}(a_n)a_nf(a_n)}.
\end{align*}
By Karamata's theorem (Theorem \ref{karamata}), the sequence $(v_n)$
converges to $1$, since $f$ belongs to $\RV_{-1-\alpha}$. One deduces
that $F/f$ belongs to $\RV_{1+\alpha}$. Since $F/f$ is bounded on
finite intervals of $\R_{\geq 0}$, the uniform convergence theorem
yields
\begin{equation*}
   L_n(x) =  \frac{F(a_nx)f(a_n)}{f(a_nx)F(a_n)} \underset{n\rightarrow +\infty}{\longrightarrow} x^{1+\alpha}
\end{equation*}
uniformly on $]0,1]$, which in turn implies that the sequence of
functions $(q_n)$ converges to $0$ uniformly on $]0,1]$. It follows
that the sequence of functions
\begin{equation*}
    v \longmapsto q_n \left ( \frac{U \left ( \frac{n}{v} \right )}{a_n} \right )
\end{equation*}
is uniformly bounded (see \cite[Proposition B.1.9
p366-367]{haan2006extreme}). We can thus use Lebesgue's dominated
convergence theorem to deduce that the first integral goes to 0 as $n
\to \infty$.

We still need to deal with the second integral.  To this end, we
introduce the notation $j(x,t) := \frac{(\log F)'(t)}{(\log
  F)'(tx)}-x^{1+\alpha} = \frac{F(tx)f(t)}{f(tx)F(t)}-x^{1+\alpha}$
and obtain
\begin{align*}
& \int_{1}^{+\infty}
                                        \left| 1 - \frac{\alpha}{x^{\alpha+1}} \frac{1}{n a_n r_F(a_nx)} \right|
                                        f_n(x)\, \dint x \\
                                    & = \int_{1}^{+\infty} \left  |\frac{1}{x^{1+\alpha}}\frac{(\log F)'(a_n)}{(\log F)'(a_nx)}\frac{\alpha \overline{F}(a_n)}{a_nf(a_n)}\frac{F(a_n)}{n\overline{F}(a_n)}-1\right |f_n(x) \, \dint x \\
                                    & = \int_{1}^{+\infty} \left  |\frac{v_n}{x^{1+\alpha}}\left ( x^{1+\alpha} + j(x,a_n) \right )-1\right |f_n(x) \, \dint x \\
                                    & =\int_{1}^{+\infty} \left | v_n - 1 + v_n\frac{j(x,a_n)}{x^{1+\alpha}} \right | f_n(x) \, \dint x  \\
                                    & \leq |v_n-1| +
                                  |v_n|    \int_{1}^{+\infty}
                                      \frac{|j(x,a_n)|}{x^{1+\alpha}}
                                      f_n(x)\, \dint x,
\end{align*}
where
$v_n = \frac{\alpha
  \overline{F}(a_n)}{a_nf(a_n)}\frac{F(a_n)}{n\overline{F}(a_n)}$. Recall
that the sequence $(v_n)$ converges to $1$. One has
\begin{equation*}
    \begin{aligned}
      |v_n| \frac{|j(x,a_n)|}{x^{1+\alpha}}f_n(x)\I_{_{[1,+\infty[}}(x) & =  \left (\frac{\alpha F(a_nx)^n}{x^{1+\alpha}} - v_nf_n(x) \right )\I_{_{[1,+\infty[}}(x) \\[2mm]
      & = \left (\frac{\alpha F(a_nx)^n}{x^{1+\alpha}} - \alpha\frac{
          f(a_nx)}{f(a_n)}F(a_nx)^{n-1} \right
      )\I_{_{[1,+\infty[}}(x).
    \end{aligned}
\end{equation*}
By Proposition \ref{Drees}, applied with $\delta = \frac{\alpha}{2}$ and $\varepsilon = 1$, there exists $t_1 > 0$ such that
\[\left | \frac{f(a_nx)}{f(a_n)} - \frac{1}{x^{1+\alpha}} \right | \leq \max(x^{-1-\alpha/2}, x^{1+3\alpha/2})\]
for every $a_nx \geq t_1$. Thus, for $n$ big enough, one reaps
\[\left | \frac{f(a_nx)}{f(a_n)} - \frac{1}{x^{1+\alpha}} \right | \leq \frac{1}{x^{1+\alpha/2}}, \]
for every $x\geq 1$. One achieves domination if $n$ is big enough,
\begin{equation*}
    \begin{aligned}
        |v_n| \frac{|j(x,a_n)|}{x^{1+\alpha}}f_n(x)\I_{_{[1,+\infty[}}(x) & \leq \left ( \frac{2\alpha}{x^{1+\alpha}} + \frac{\alpha}{x^{1+\frac{\alpha}{2}}} \right ) \I_{_{[1,+\infty[}}(x).
    \end{aligned}
\end{equation*}
The conclusion follows from Lebesgue's dominated convergence theorem.

Finally to conclude our proof of Theorem \ref{borneXalpha} we need to show that $\Delta_w(\Phi_\alpha \mid F_n)$ (resp., $\Delta_w(F_n \mid \Phi_\alpha)$) goes to 0 under the stated conditions. The original proofs go through basically unchanged. Indeed, for the first item, the bound  \eqref{part1aless1} remains (multiplied by $x$), and one only needs to deal with \eqref{auxaux}. Since $\alpha>1$, that can be accomplished by taking $\varepsilon = 1$ and $\delta = (\alpha-1)/2$ in Proposition \ref{Drees}. Indeed, that leads to domination
 \begin{align}\label{ineq:dom_lipschitz}
     |g_F(x,t)| & \leq \left ( x + Mx^{1/2} \right ) \I_{[0,1]}(x) + x^{(1+\alpha)/2}\I_{[1,+\infty[}(x)
 \end{align}
 for $t$ big enough. Then, the conclusion unfolds in the same way.
 For the second item, we still break the integral and get
 \begin{align*}
   \Delta_w(F_n \mid \Phi_{\alpha})   & = \int_{c_F/a_n}^1
                                      |x|\left| 1 - \frac{\alpha}{x^{\alpha+1}} \frac{1}{n a_n r_F(a_nx)} \right|
                                      f_n(x)\, \dint x \\
                                      & \quad + \int_{1}^{+\infty}
                                      |x|\left| 1 - \frac{\alpha}{x^{\alpha+1}} \frac{1}{n a_n r_F(a_nx)} \right|
                                      f_n(x)\, \dint x.
                                       \\
                                       & =: I_1 + I_2.
\end{align*}
On the one hand, since $\forall x \in ]c_F/a_n, 1], \; |x| <1$, one bounds $I_1$ by the integral
\[\int_{c_F/a_n}^1     \left| 1 - \frac{\alpha}{x^{\alpha+1}} \frac{1}{na_nr_F(a_nx)} \right| f_n(x)\, \dint x\]
which, from the proof above, goes to $0$ as $n$ goes to $+\infty$. On the other hand, we analyse $I_2$ as before, except that we apply Proposition \ref{Drees} with $\delta = \frac{\alpha-1}{2}$ (and $\varepsilon = 1$) this time and we get, if $n$ is large enough,
\begin{align*}
        |x||v_n| \frac{|j(x,a_n)|}{x^{1+\alpha}}f_n(x)\I_{_{[1,+\infty[}}(x) & \leq \left ( \frac{2\alpha}{x^{\alpha}} + \frac{\alpha}{x^{ \frac{\alpha+1}{2}}} \right ) \I_{_{[1,+\infty[}}(x),
\end{align*}
for all $x \in \R$. Since $\alpha > 1$, Lebesgue's dominated convergence theorem concludes again.

\backmatter

\bmhead{Declaration on AI-assisted language editing and code preparation}

During the preparation of the revision of this manuscript, the authors used ChatGPT
(OpenAI, GPT-5.5 Thinking, accessed in June 2026) for language editing,
clarity improvements, and assistance in drafting and debugging the simulation
code used in Examples~\ref{ex:bootstrap-diagnostic} and~\ref{ex:estimation}.
All AI-assisted output was reviewed, checked, and edited by the authors, who
take full responsibility for the final content of the manuscript. The simulation
codes are provided in the supplementary material~\cite{mansanonline}.

\bmhead{Acknowledgements}
The authors thank the anonymous referees and the Associate Editor for their
careful reading and constructive comments. In particular, the authors are
grateful to one anonymous referee for suggesting the counterexample presented in
Example~\ref{ex:acounter}, which in turn motivated the diagnostic procedure
proposed in Example~\ref{ex:bootstrap-diagnostic}. Paul Mansanarez is funded by the French
Community of Belgium through a FRIA grant from the FNRS (Fonds de la Recherche
Scientifique). Yvik Swan is funded in part by FNRS grant J.0200.24F.

\begin{appendices}

\section{Regularly varying functions}
 \label{sec:regul-vary-funct}

 As anticipated, in this section, we present the results on regularly
 varying functions (at infinity) that were used in the proofs. Most of
 the results are taken from the appendix of the classical reference
 \cite{haan2006extreme}, but there are usually due to others authors. We provide quick proofs for the results that
 are not directly in the literature.

 For a non-decreasing function $U$, we will
 denote by $U^{\leftarrow}$ its generalized inverse function, defined
 by
\begin{align*}
  U^{\leftarrow}(t) := \inf \big \{ x \in \R \mid  U(x) > t \big \}.
\end{align*}
Whenever $U$ is \rev{continuous and increasing}, one has $U^{\leftarrow} = U^{-1}$. As
mentioned in the main text, being in a domain of attraction of the
Fr\'echet distribution of parameter $\alpha$ is related to knowing
whether the survival function $1-F$ belongs to the class of
regularly varying functions of parameter $-\alpha$.

\begin{definition}[B.1.1, \cite{haan2006extreme}] \label{sec:regul-vary-funct-1}
    Let $f: \R_{\geq 0} \longrightarrow \R$ be a Lebesgue measurable
    function, eventually positive and let $\rho$ be in $\R$. One says
    that $f$ is {\it regularly varying of index $\rho$} if
    \begin{equation*}
        \underset{t\rightarrow +\infty}{\lim} \, \frac{f(tx)}{f(t)} = x^\rho,
    \end{equation*}
    for every $x$ in $\R_{>0}$. One denotes by $\RV_\rho$ the set of
    all regularly varying functions of index $\rho$.
\end{definition}
\begin{remark}
  The function $f$ above only needs to be defined on a neighborhood of
  $+\infty$. In that case, one can always extend it by $0$ on the
  remainder of $\R_{\geq 0}$. \end{remark}

\rev{The connection between regular variation and the Fr\'echet DA is summarized in the next result taken from \cite[Theorem~1.2.1.1]{haan2006extreme} but originally due to Gnedenko (see \cite{Gnedenko1943}).}

\begin{proposition}[\rev{Characterization of Fr\'{e}chet domain of attraction}]
  Let $X$ be a random variable with cumulative distribution function
  $F$ and let $\alpha$ be a positive real number. Then $X$ is in the
  domain of attraction of $\Phi_\alpha$ if and only if for every
  $x \in \R$, $F(x) < 1$ and $1-F \in \RV_{-\alpha}$. Furthermore, one
  can choose $a_n = F^{\leftarrow}(1-1/n)$ and $b_n=0$.
  \end{proposition}

  \rev{The following result, known as Karamata's theorem (see \cite{Karamata1930} and also \cite[Proposition~1.5.8]{Bingham1987}), is key. We state the version \cite[Theorem~B.1.5]{haan2006extreme}.}

  \begin{theorem}[Karamata's Theorem]  \label{karamata} Let $f \in RV_\alpha$ for
    some $\alpha$ in $\R$. Then there exists $t_0>0$ such that $f$ is
    positive and locally bounded on $[t_0,+\infty[$. If $\alpha > -1$
    then
    \begin{equation}\label{-1alpha}
      \underset{t\rightarrow +\infty}{\lim} \, \frac{tf(t)}{\int_0^t f(s) \, d s} = \alpha + 1.
    \end{equation}
    If $\alpha \leq -1$ and $\int_0^{+\infty} f(s) \, d s < +\infty$, then
    \begin{equation}\label{alpha-1}
      \underset{t\rightarrow +\infty}{\lim} \, \frac{tf(t)}{\int_t^{+\infty} f(s) \, d s} = -\alpha - 1.
    \end{equation}
    Conversely, if \eqref{-1alpha} holds for $\rho>-1$ then
    $f \in \RV_\rho$, and if \eqref{alpha-1} holds for $\rho < -1$
    then $f \in \RV_\rho$.
\end{theorem}

Let $f$ be a continuous density function such that $f$ is in $\RV_{-1-\alpha}$ for some $\alpha>0$. Then by Theorem~\ref{karamata}, one gets
\begin{align}
    \underset{t\rightarrow +\infty}{\lim} \, \frac{tf(t)}{1-F(t)} = -(-1-\alpha) - 1 =\alpha,
\end{align}
where $F(t) := \int_{-\infty}^t f(u) \, d u$ is the cumulative
distribution function associated to $f$.
\begin{proposition}
    \label{HopToRV}
    Let $f$ be a density function on $\R$, belonging to
    $\RV_{-1-\alpha}$ for some $\alpha>0$. Let $F$ be its cumulative
    distribution function and set $L := f/F$. Then $L$ belongs to
    $\RV_{-1-\alpha}$.
\end{proposition}

\begin{proof}
    Indeed, since $\underset{u\rightarrow +\infty}{\lim}\, F(u) = 1$ and $f \in \RV_{-1-\alpha}$, one gets, for $x>0$,
    \begin{align}
        \frac{L(tx)}{L(t)} = \frac{f(tx)}{f(t)}\frac{F(t)}{F(tx)} \underset{t\rightarrow+\infty}{\longrightarrow} x^{-1-\alpha}.
    \end{align}
\end{proof}
\rev{The next result is known as Potter's theorem (or an application of, see e.g. \cite{Potter1942} or \cite[Theorem 1.5.6]{Bingham1987}). We state the version \cite[Proposition B.1.9, 7. ]{haan2006extreme}.}
\begin{proposition}[Potter's Theorem] \label{gammaleq0}
    If $f$ is in $RV_{\rho}$ with $\rho \leq 0$ and is bounded on finite
    intervals of $\R_{\geq 0}$, then for every $\delta > 0 $ and every
    $\xi>0$, there exists $c>0$ and $t_0 >0$ such that for $t\geq t_0$
    and $0<x\leq \xi$,
    \[\frac{f(tx)}{f(t)} \leq c_{\delta,\xi}x^{\rho-\delta}.\]
\end{proposition}
That means that for $0<\varepsilon <1$, there exists $t_0>0$ and
$M_{\varepsilon} >0$ such that
\begin{equation}\label{fnear0}
    x^{1+\alpha}\frac{f(tx)}{f(t)} \leq M_{\varepsilon}x^{-\varepsilon},
\end{equation}
for every $0<x\leq 1$ and $t\geq t_0$.

\rev{The next result can be both seen as again Potter's theorem, although it is a stronger version, or also as a first order version of Drees' Lemma (see e.g. \cite{Drees1998}). We state the version \cite[Proposition B.1.10]{haan2006extreme}.}
\begin{proposition}\label{Drees} Let
  $f$ be in $\RV_{\alpha}$. Then for every $\delta >0$ and
  $\varepsilon>0$, there exists $t_0 >0$ such that for every $t>0$ and
  every $x>0$ such that $tx \geq t_0$,
    \[\left | \frac{f(tx)}{f(t)} - x^{\alpha} \right | \leq
      \varepsilon \max (x^{\alpha+\delta},x^{\alpha-\delta}).\]
\end{proposition}

\rev{ The next result is the well-known \emph{Uniform Convergence Theorem} due to Karamata (see, e.g., \cite{Karamata1930,Karamata1933}). We state the version \cite[Theorem 1.5.2]{Bingham1987}].
\begin{proposition}[Uniform Convergence Theorem]\label{UnifConv}
    Let $f$ be an element of $\RV_\rho$ with $\rho>0$, bounded on each
    interval $]0,b]$ with $b>0$. Then the convergence
    \begin{align*}
        \underset{t \rightarrow +\infty}{\lim} \, \frac{f(tx)}{f(t)} = x^\rho
    \end{align*}
    is uniform in $x$ on each interval $]0,b]$, with $b>0$.
  \end{proposition}
  }

\rev{The next three last results are well-known properties of regularly varying functions and generalized inverses. We state them in our context, and give small proofs for two of them.

\begin{proposition}[B.1.9, 4, \cite{haan2006extreme}]
  Let $\rho<-1$ and let $f$ be a continuous probability density
  function in $\RV_{\rho}$. Then its cumulative distribution function
  $F$ satisfies $1-F \in \RV_{\rho +1}$.
\end{proposition}

\begin{proposition}\label{FtoU}
  Let $F$ be a cumulative distribution function on $\R$ such that
  $1-F$ belongs in $\RV_{-\alpha}$ and $F(x)<1$ for every $x>0$, with
  $\alpha>0$. Then the function
  $U := \left (\frac{1}{1-F} \right)^{\leftarrow}$ belongs in
  $\RV_{1/\alpha}$.
\end{proposition}
\begin{proof}
    Since $1-F \in \RV_{-\alpha}$, for all $x >0$, one has
    \begin{align*}
        \underset{t\rightarrow +\infty}{\lim} \, \frac{1-F(t)}{1-F(tx)} = \underset{t\rightarrow +\infty}{\lim} \, \left (\frac{1-F(tx)}{1-F(t)} \right )^{-1} = \left ( \frac{1}{x^{\alpha}} \right )^{-1} = x^{\alpha},
    \end{align*}
    that is, $1/(1-F) \in \RV_{\alpha}$. Then, since $\alpha>0$ and
    $1/(1-F)$ is increasing, by \cite[Proposition
    B.1.9, 9]{haan2006extreme}, $U$ belongs to $\RV_{1/\alpha}$.
\end{proof}

\begin{proposition}\label{Fto1overU}
    Let $F$ be a cumulative distribution function on $\R$ such that
    $1-F$ belongs in $\RV_{-\alpha}$ and $F(x)<1$ for every $x>0$,
    with $\alpha>0$. Suppose that the support of $F$ is $[a,+\infty[$
    with $a>0$. Denote by $U$ the function $\left (\frac{1}{1-F}
    \right)^{\leftarrow}$. Then $1/U$ belongs in $\RV_{-1/\alpha}$.
\end{proposition}
\begin{proof}
    Since $\lim_{t\rightarrow +\infty} \frac{1}{1-F(t)} = +\infty$, one has $\lim_{t\rightarrow+\infty} U(t) = +\infty$. By \cite[Proposition B.1.9,2]{haan2006extreme}, one deduces $1/U \in \RV_{-1/\alpha}$.
\end{proof}
}

\section{Computations}
\label{sn:appendixfurther}

\subsection{Computations for Example \ref{ex:not-equivalence}}
\label{sn:not-SORV}
We prove here that the cdf $F$ constructed in Example \ref{ex:not-equivalence} does not verify the second-order regular variation condition (SORVC). First, recall that the SORVC requires the existence of a function $A(t) \to 0$, in $\RV_\rho$ for some $\rho \leq 0$, such that
\begin{equation}\label{SORVC.F}
    \frac{1}{A(t)}\left(\frac{\bar{F}(tx)}{\bar{F}(t)} - x^{-\alpha}\right) \; \xrightarrow[t \to +\infty]{} \;  x^{-\alpha} h_\rho(x), \qquad ,
\end{equation}
for every $x > 0$, where $h_\rho(x) = (x^\rho - 1)/\rho$ for $\rho < 0$ and
$h_0(x) = \log x$. We prove that \eqref{SORVC.F} does not hold for any choice of $A$.

Consider $x_0 = \tfrac{3}{2}$ and the sequence $(t_n)_{n\geq 1}$ defined by $\forall n \in_{\geq 1}, \; t_n := 2n$. On one hand, since $t_n x_0 = 3n$, we have
\begin{align*}
    \frac{\bar{F}(t_n x_0)}{\bar{F}(t_n)} = \frac{(3n)^{-\alpha}}{(2n)^{-\alpha}} = \left ( \frac{3}{2} \right )^{-\alpha} \, .
\end{align*}
On the other hand, one has 
\begin{align*}
    h_\rho(x_0) = \begin{cases}
        \log \left ( \tfrac{3}{2} \right ) & \text{ if } \rho = 0 \; ; \\ \\
        \frac{\left ( \tfrac{3}{2} \right )^{\rho}-1}{\rho} & \text{ if } \rho <0 \; .
    \end{cases}
\end{align*}
so that, in either cases, $h_\rho(x_0) >0$. Since
\begin{equation*}
    \forall n \in \N_{> 0}, \quad \frac{1}{A(t_n)}\left ( \frac{\bar{F}(t_n x_0)}{\bar{F}(t_n)} - x_0^{-\alpha} \right ) = 0 \, ,
\end{equation*}
one cannot have 
\begin{align*}
    \frac{1}{A(t_n)}\left(\frac{\bar{F}(t_nx_0)}{\bar{F}(t_n)} - x_0^{-\alpha}\right) \; \xrightarrow[n \to +\infty]{} \;  x_0^{-\alpha} h_\rho(x_0)\,.
\end{align*}
Thus $F$ does not verify the second-order regular variation condition of \cite{Resnick1996}.

\subsection{Computations for Example \ref{sec:cauchy}}
\label{sn:appendccaychy}

 In this case
\[
\Delta(\Phi_1\mid F_n)
=
\int_0^\infty
\left|
1-
\frac{x^2n^2}
{(\pi^2+n^2x^2)
\left(
\frac12+\frac1\pi\arctan\left(\frac{nx}{\pi}\right)
\right)}
\right|
\frac{e^{-1/x}}{x^2}\,\dint x.
\]
After the change of variables \(u=(nx)^{-1}\), this discrepancy can be
rewritten as
\[
  \Delta(\Phi_1\mid F_n)
  =
  n\int_0^\infty |1-R(u)|e^{-nu}\,\dint u
  =
  \mathbb{E}\bigl[|1-R(E_n)|\bigr],
\]
where \(E_n\sim\mathrm{Exp}(n)\) and
\[
  R(u)
  :=
  \frac{1}
  {(1+\pi^2u^2)
  \left(1-\frac1\pi\arctan(\pi u)\right)}.
\]
One now  checks that $$\lim_{u\to0^+}R(u)=1, \quad 1-R(u)\sim -u \mbox{ as }u\to0^+$$
and  
$$\lim_{u\to+\infty}R(u)=0.$$   
Also  there exists a unique \(z_0\in(0,1)\)  such that \(R(u)>1\) on \((0,z_0)\) and \(R(u)<1\) on \((z_0,\infty)\); numerically, \(z_0\simeq 0.109\).
Since \(R\) is smooth, integration by parts yields
\[
\begin{aligned}
& n\int_0^\infty |1-R(u)|e^{-nu}\,\dint u \\
&=
\int_0^{z_0} R'(u)e^{-nu}\,\dint u
-
\int_{z_0}^{\infty} R'(u)e^{-nu}\,\dint u \\
&=
\frac{1}{n}
-\frac{2}{n}R'(z_0)e^{-nz_0}
+\frac{1}{n}\int_0^{z_0}R''(u)e^{-nu}\,\dint u
-\frac{1}{n}\int_{z_0}^{\infty}R''(u)e^{-nu}\,\dint u .
\end{aligned}
\]
One checks that  $\|R''\|_{\infty,[0,z_0]}<\infty, $  and 
$\int_{z_0}^\infty |R''(u)|\,\dint u<\infty.$
Since 
\[
  \frac1n
  \left|
  \int_0^{z_0}R''(u)e^{-nu}\,\dint u
  \right|
  \le
  \|R''\|_{\infty,[0,z_0]}\frac{1-e^{-nz_0}}{n^2}
\]
and
\[
  \frac1n
  \left|
  \int_{z_0}^\infty R''(u)e^{-nu}\,\dint u
  \right|
  \le
  \frac{e^{-nz_0}}{n}
  \int_{z_0}^\infty |R''(u)|\,\dint u
\]
it follows that
\[
  \Delta(\Phi_1\mid F_n)
  \le
  \frac{C}{n},
\]
with
\[
C
=
1
+
2|R'(z_0)|
+
\|R''\|_{\infty,[0,z_0]}
+
\int_{z_0}^\infty |R''(u)|\,\dint u .
\]
Numerical estimation gives
\[
  |R'(z_0)| \lesssim 0.93,
  \qquad
  \|R''\|_{\infty,[0,z_0]}\lesssim 19.02,
  \qquad
  \int_{z_0}^\infty |R''(u)|\,\dint u\lesssim 2.64.
\]
Thus one may take \(C=25\).

Similar conclusions hold for \(\Delta_w(\Phi_1\mid F_n)\), with different
constants. Indeed the same change of variables  as above gives
\[
  \Delta_w(\Phi_1\mid F_n)
  =
  \int_0^\infty \frac{|1-R(u)|}{u}e^{-nu}\,\dint u.
\]
One checks that 
\[
  C_w:=\sup_{u\geq0}\frac{|1-R(u)|}{u}< 2,
\]
hence 
\[
  \Delta_w(\Phi_1\mid F_n)
  \le
  C_w\int_0^\infty e^{-nu}\,\dint u
  < 
  \frac{2}{n}.
\]
The numerical evaluations can be found in the  supplementary material \citep{mansanonline}.

\subsection{Computations for Example \ref{sec:burr}}
\label{sn:exampburr}

A direct computation gives
\[
\begin{aligned}
  \Delta(\Phi_{\alpha\tau}\mid F_n)
  &=
  \int_0^\infty
  \left|
  1-
  \frac{
    n^{1+1/\tau}x^{\alpha(1+\tau)}
  }
  {
    (1+n^{1/\tau}x^\alpha)
    \left((1+n^{1/\tau}x^\alpha)^\tau-1\right)
  }
  \right| \\
  &\qquad\qquad\times
  (\alpha\tau)x^{-(1+\alpha\tau)}
  e^{-x^{-\alpha\tau}}
  \,\dint x .
\end{aligned}
\]
With the change of variables $y=n^{1/\tau}x^\alpha$, this becomes
\[
  \Delta(\Phi_{\alpha\tau}\mid F_n)
  =
  \int_0^\infty |1-R(y)|\,
  n\tau y^{-(\tau+1)}e^{-n/y^\tau}\,\dint y
  =
  \mathbb{E}\left[
  \left|1-R(E_n^{-1/\tau})\right|
  \right],
\]
where $E_n\sim\mathrm{Exp}(n)$ and
\[
  R(y)
  :=
  \frac{y^{1+\tau}}
  {(1+y)\left((1+y)^\tau-1\right)}.
\]
Since $\tau>1$, one has $R(y)\le1$ for all $y>0$.  Moreover,
\[
  R(y)
  =
  1-\frac{\tau+1}{y}
  +O\left(y^{-\min(2,\tau)}\right),
  \qquad y\to\infty.
\]
Therefore,
\[
  1-R(E_n^{-1/\tau})
  =
  (\tau+1)E_n^{1/\tau}
  +O\left(E_n^{\min(2,\tau)/\tau}\right),
\]
and hence
\[
  \Delta(\Phi_{\alpha\tau}\mid F_n)
  \sim
  (\tau+1)
  \mathbb{E}\bigl[E_n^{1/\tau}\bigr]
  =
  (\tau+1)
  \Gamma\left(1+\frac1\tau\right)n^{-1/\tau}.
\]
This agrees with the prediction obtained from the second-order tail
expansion.

\subsection{Computations for Example \ref{ex:acounter}}
\label{ssn:proofexacounter}

We have
\[
\Delta(F_n\mid \Phi_\alpha)
=
\int_{1/a_n}^{\infty}
\left|
1-
\frac{\alpha}{x^{\alpha+1}}
\frac{1}{n a_n r_F(a_nx)}
\right| f_n(x)\,\dint x .
\]
With \(y=a_nx\) and \(a_n^\alpha=n\), this becomes
\[
\Delta(F_n\mid \Phi_\alpha)
=
\int_1^\infty
\left|
1-
\frac{\alpha F(y)}{y^{\alpha+1}f(y)}
\right|
n f(y)F(y)^{n-1}\,\dint y .
\]

On \([k,k+1)\), set
$m_k=k^{-\alpha}-(k+1)^{-\alpha}.$
Then \(m_k\sim \alpha k^{-\alpha-1}\) and
\[
f(y)=
\begin{cases}
\dfrac32 m_k, & y\in [k,k+\frac12),\\[4pt]
\dfrac12 m_k, & y\in [k+\frac12,k+1).
\end{cases}
\]
Hence, uniformly for \(y\in[k,k+1)\),
\[
y^{\alpha+1}m_k\to \alpha .
\]
Moreover, under the probability measure
\[
\mu_n(\dint y)=n f(y)F(y)^{n-1}\,\dint y,
\]
the variable \(y\) tends to infinity in probability. Therefore \(F(y)\to1\) on
the relevant range of integration, and the ratio
\[
\frac{\alpha F(y)}{y^{\alpha+1}f(y)}
\]
tends to \(2/3\) on the first half of each interval and to \(2\) on the second
half. The corresponding absolute errors are \(1/3\) and \(1\).

It remains to average these two local values. The first half of \([k,k+1)\)
carries mass \((3/4)m_k\), while the second half carries mass \((1/4)m_k\).
Since on the scale \(k\asymp n^{1/\alpha}\) one has \(n m_k\to0\), the factor
\(F(y)^{n-1}\) is asymptotically constant over each unit interval. Thus the
limiting local weights are \(3/4\) and \(1/4\), and
\[
\Delta(F_n\mid \Phi_\alpha)
\longrightarrow
\frac34\cdot\frac13+\frac14\cdot1
=
\frac12 .
\]

\end{appendices}


\begin{thebibliography}{10}

\bibitem[Bartholm\'e and Swan(2013)]{bartholme2013rates}
Carine Bartholm{\'e} and Yvik Swan.
\newblock Rates of convergence towards the Fr\'echet distribution.
\newblock {\em arXiv preprint arXiv:1311.3896}, 2013.

\bibitem[Ben Moussa et~al.(in preparation)]{benmou25}
Thomas Ben~Moussa, Paul Mansanarez,  and Yvik Swan.
\newblock Rates of convergence towards GEV distributions: a craftsman's
approach.
\newblock In preparation.

\bibitem[Bingham et~al.(1987)]{Bingham1987}
Nicholas~H. Bingham, Charles~M. Goldie, and Jozef~L. Teugels.
\newblock {\em Regular Variation}.
\newblock Encyclopedia of Mathematics and its Applications. Cambridge
University Press, Cambridge, 1987.

\bibitem[Bobbia et~al.(2021)]{BobbiaDombryVarron2021}
Benjamin Bobbia, Cl\'ement Dombry, and Davit Varron.
\newblock The coupling method in extreme value theory.
\newblock {\em Bernoulli}, 27(3):1824--1850, 2021.

\bibitem[Cheng and Jiang(2001)]{Cheng2001}
Shihong Cheng and Changguo Jiang.
\newblock The Edgeworth expansion for distributions of extreme values.
\newblock {\em Science in China Series A: Mathematics}, 44(4):427--437, 2001.

\bibitem[Costac\`eque\mbox{-}Cecchi(2024)]{costaceque2024stein}
Bruno Costac{\`e}que.
\newblock {\em Stein's method for extreme value distributions}.
\newblock PhD thesis, Institut Polytechnique de Paris, 2024.

\bibitem[Costac\`eque and Decreusefond(2025)]{CostacequeDecreusefond2025}
Bruno Costac\`eque and Laurent Decreusefond.
\newblock Stein's method for max-stable random vectors.
\newblock {\em arXiv preprint arXiv:2507.00463}, 2025.

\bibitem[de Haan and Ferreira(2006)]{haan2006extreme}
Laurens de~Haan and Ana Ferreira.
\newblock {\em Extreme value theory: an introduction}, volume~3.
\newblock Springer, 2006.

\bibitem[Drees(1998)]{Drees1998}
H. Drees,
On Smooth Statistical Tail Functionals,
\emph{Scandinavian Journal of Statistics},
\textbf{25} (1998), no. 1, 187--210.

\bibitem[Ernst et~al.(2020)]{Ernst2020}
Marie Ernst, Gesine Reinert, and Yvik Swan.
\newblock First-order covariance inequalities via Stein's method.
\newblock {\em Bernoulli}, 26(3):2051--2081, 2020.

\bibitem[Ernst and Swan(2022)]{Ernst2022}
Marie Ernst and Yvik Swan.
\newblock Distances between distributions via Stein's method.
\newblock {\em Journal of Theoretical Probability}, 35(2):949--987, 2022.

\bibitem[Feidt(2013)]{feidt2013stein}
Anne Feidt.
\newblock Stein's method for multivariate extremes.
\newblock {\em arXiv preprint arXiv:1310.2564}, 2013.



\bibitem[Ebner et al.(2025)]{eetal2025}
Bruno Ebner, Adrian Fischer, Robert E. Gaunt, Babette Picker, and Yvik Swan. \newblock Stein's method of moments. 
\newblock{\em Scandinavian Journal of Statistics},  52(4):1594-1624, 2025.


\bibitem[Gnedenko(1943)]{Gnedenko1943}
B. Gnedenko,
Sur la distribution limite du terme maximum d'une série aléatoire,
\emph{Annals of Mathematics},
2nd Series, \textbf{44} (1943), no. 3, 423--453.

\bibitem[Karamata (1930)]{Karamata1930}
J. Karamata,
Sur un mode de croissance régulière des fonctions,
\emph{Mathematica (Cluj)},
4 (1930), 38--53.

\bibitem[Karamata (1933)]{Karamata1933}
J. Karamata,
Sur un mode de croissance régulière.
Théorèmes fondamentaux,
\emph{Bulletin de la Société Mathématique de France},
61 (1933), 55--62.


\bibitem[Kusumoto and Takeuchi(2020)]{Kusumoto2020}
Hideaki Kusumoto and Atsushi Takeuchi.
\newblock Remark on rates of convergence to extreme value distributions via the
Stein equations.
\newblock {\em Extremes}, 23(3):411--423, 2020.

\bibitem[Ley et~al.(2017)]{ley2017stein}
Christophe Ley, Gesine Reinert, and Yvik Swan.
\newblock Stein's method for comparison of univariate distributions.
\newblock {\em Probability Surveys}, 14:1--52, 2017.


\bibitem[Mansanarez(2026)]{mansanonline} 
Paul Mansanarez, Guillaume  Poly, and  Yvik Swan. 
\newblock Mathematica notebook. 
\newblock https://www.wolframcloud.com/obj/yvik.swan/Published/MPS26Num.nb

\bibitem[Potter (1942)]{Potter1942}
H. S. A. Potter,
The mean values of certain Dirichlet series. II,
\emph{Proceedings of the London Mathematical Society},
Ser. 2, \textbf{47} (1942), 1--19.

\bibitem[Resnick(2008)]{resnick2008extreme}
Sidney Resnick.
\newblock {\em Extreme values, regular variation, and point processes},
volume~4.
\newblock Springer Science \& Business Media, 2008.

\bibitem[Resnick and de Haan(1996)]{Resnick1996}
Sidney Resnick and Laurens de~Haan.
\newblock Second-order regular variation and rates of convergence in
extreme-value theory.
\newblock {\em The Annals of Probability}, 24(1):97--124, 1996.

\bibitem[Smith(1982)]{smith1982uniform}
Richard L. Smith.
\newblock Uniform rates of convergence in extreme-value theory.
\newblock {\em Advances in Applied Probability}, 14(3):600--622, 1982.









\end{thebibliography}
\end{document}